\documentclass[10pt]{article}
\usepackage{amsfonts}
\usepackage{amsmath,amsthm,amssymb,latexsym,amscd,xcolor}
\usepackage{epsf}
\usepackage[pdftex]{graphicx}
\usepackage{tikz}
\usepackage{hyperref} 
\newtheorem{theorem}{Theorem}[section]
\newtheorem{lemma}[theorem]{Lemma}
\newtheorem{corollary}[theorem]{Corollary}

\theoremstyle{definition}
\newtheorem{definition}[theorem]{Definition}

\theoremstyle{remark}
\newtheorem{remark}[theorem]{Remark}

\numberwithin{equation}{section}

\title{\bf  Creating band gaps in periodic media}

\author{Robert Lipton\thanks{Department of Mathematics, Louisiana State University,
Baton Rouge, LA 70803, USA,
{\tt lipton@math.lsu.edu}}
\and
Robert Viator Jr.
\thanks{Department of Mathematics,
Louisiana State University,
Baton Rouge, LA 70803, USA,
{\tt rviato2@lsu.edu}}}

\date{}
\begin{document}
\maketitle
\begin{abstract}

We identify explicit conditions on geometry and material contrast for creating band gaps in 2-d photonic and 3-d acoustic crystals. This approach is novel and makes use of the electrostatic and quasi-periodic source free resonances of the crystal. The source free modes deliver a spectral representation for solution operators associated with propagation of electromagnetic and acoustic waves inside periodic high contrast media. An accurate characterization of the quasi-periodic and electrostatic resonance spectrum  in terms of the shape and geometry of the scatters is possible. This information together with the Dirichlet and a Neumann like spectra associated with the inclusions provide conditions sufficent  for opening band gaps at finite contrast. The theory provides a systematic means for the identification of photonic and phononic band gaps within a specified frequency range.

\end{abstract}

%
%
\maketitle

\section{Introduction}
\label{introduction}

High contrast periodic media are known to exhibit unique optical, acoustic, and elastic properties \cite{J87}, \cite{Y}.  In this paper we identify new explicit conditions on geometry and material contrast for creating band gaps in 2-d photonic and 3-d acoustic crystals. We consider wave propagation through a periodic  medium in $\mathbb{R}^d$, $d=2,3$, made from  two materials. One of the materials is in the form of disjoint inclusions. The inclusions  are completely surrounded by the second material and do not touch the boundary of the period cell.  The material coefficient is taken to be $1$ inside the inclusions and and takes the value $k>1$ in the surrounding material. The union of all the inclusions $D_1,D_2,\ldots,D_n$ inside each period is denoted by $D$, see figure \ref{plane}.  The crystal occupies $\mathbb{R}^d$ and is described by the periodic array of inclusions  $\Omega=\cup_{m\in\mathbb{Z}^d}(D+m)$, with fundamental period cell $Y=(0,1]^d$. The material coefficient for the medium is written
$a(x)=k(1- \chi_{\scriptscriptstyle{\Omega}}(x))+ \chi_{\scriptscriptstyle{\Omega}}(x)$ where $\chi_{\scriptscriptstyle{\Omega}}$ is the indicator function for $\Omega$ taking the value $1$ inside $\Omega$ and zero outside. 

Wave propagation inside the crystal at frequency $\omega$ is described by the spectral problem
\begin{equation}
-\nabla\cdot(k(1- \chi_{\scriptscriptstyle{\Omega}}(x))+ \chi_{\scriptscriptstyle{\Omega}}(x))\nabla u(x) =\omega^2 u(x) ,\hbox{  $x\in\mathbb{R}^d$, $d=2,3$}
\label{Eigen0}
\end{equation}
Here the self-adjoint divergence form operator $L_k=-\nabla\cdot(k(1- \chi_{\scriptscriptstyle{\Omega}})+ \chi_{\scriptscriptstyle{\Omega}})\nabla$ is defined by the quadratic form in $L^2(\mathbb{R}^d)$, $d=2, 3$
\begin{equation}
\int_{\mathbb{R}^d}a(x)|\nabla u(x)|^2\,dx
\label{quadform}
\end{equation}
with domain $W^{1,2}(\mathbb{R}^d)$, $d=2,3$.  

This mathematical formulation describes wave propagation in both two and three dimensional acoustic crystals and electromagnetic wave propagation through two dimensional photonic crystals. For acoustic wave propagation the material coefficient $a^{-1}(x)=\rho(x)$ describes the mass density $\rho(x)$ of the periodic medium. For a two dimensional photonic crystal  $a^{-1}(x)=\epsilon(x)$ describes the dielectric constant of a non-magnetic medium given by a lattice of infinitely long parallel  rods periodically arranged in the plane transverse to the long axis of the rods. The electromagnetic wave travels along the transverse plane with magnetic field directed along the rods and the electric field in the plane.

Floquet theory \cite{ReedSimon}, \cite{Kuchment}, \cite{Wilcox}, \cite{OdehKeller}  shows that the spectrum $\sigma(L_k)$ has the band structure 
\begin{equation}
\sigma(L_k)=\cup_{\scriptscriptstyle{j\in\mathbb{N}}}S_j,
\label{bandofbands}
\end{equation}
where $S_j$ are the spectral bands associated with Bloch waves propagating inside the crystal.
The Bloch waves $h(x)$ satisfy 
\begin{equation}
 -\nabla\cdot(k(1- \chi_{\scriptscriptstyle{\Omega}}(x))+ \chi_{\scriptscriptstyle{\Omega}}(x))\nabla h(x) =\omega^2 h(x) ,\hbox{  $x\in\mathbb{R}^d$, $d=2,3$}
\label{Eigen1}
\end{equation}
\linebreak
together with the $\alpha$ quasi-periodicity condition $h(x+p)=h(x)e^{i\alpha\cdot p}$. Here the wave vector $\alpha$ lies in the first Brillouin zone of the reciprocal lattice given by $Y^\star=(-\pi,\pi]^d$.  For each $\alpha\in Y^\star$ the Bloch eigenvalues $\omega^2$ are of finite multiplicity and denoted by $\lambda_j(k,\alpha)$  with $\lambda_j(k,\alpha)\leq\lambda_{j+1}(k,\alpha)$, $j\in \mathbb{N}$.

The band structure for the crystal is described by the family of dispersion relations
\begin{equation}
\omega^2=\lambda_j(k,\alpha),\hbox{ $j\in\mathbb{N}$, $\alpha\in Y^\star$}
\label{DispersionRelationsglobal}
\end{equation}
and the spectral bands are given by the intervals
\begin{equation}
S_j=[\min_{\alpha\in Y^\star}\lambda_j(k,\alpha),\max_{\alpha\in Y^\star}\lambda_j(k,\alpha)].
\label{specbandupdown}
\end{equation}

\begin{figure} 
\centering
\begin{tikzpicture}[xscale=1.0,yscale=1.0]
\draw [thick] (-2,-2) rectangle (3,3);
\draw [fill=orange,thick] (-0.2,-0.6) circle [radius=1.25];
\draw [fill=orange,thick] (2.2,2.0) circle [radius=0.6];
\node [right] at (1.9,2.0) {$D_2$};
\draw [fill=orange,thick] (-1,2.2) circle [radius=0.65];
\node [right] at (-1.3,2.2) {$D_6$};
\draw[fill=orange,thick](-1.2,1.05) ellipse (20pt and 10pt);
\node [right] at (-1.5,1.05) {$D_5$};
\draw[fill=orange,thick](2.2,-0.1) ellipse (10pt and 20pt);
\node [right] at (1.85,-0.1) {$D_4$};
\node [below] at (-0.12,-0.3) {$D_1$};
\node [below] at (0.2,2.0) {$Y\setminus D$};
\draw [fill=orange, thick] plot[ smooth cycle] coordinates{(1,0) (2,1) (1,2) (1.1,1.1)};
\node [right] at (1.1,1.0) {$D_3$};
\end{tikzpicture} 
\caption{{\bf Period Cell.}}
 \label{plane}
\end{figure}
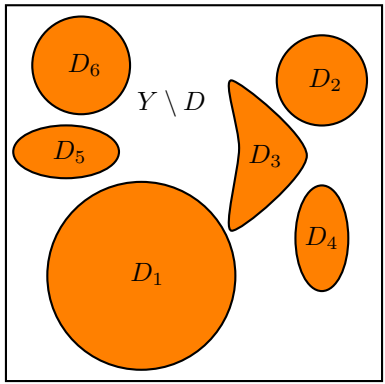

\noindent The upper and lower band edges of $S_j$ are denoted by 
\begin{equation}
b_j=\max_{\alpha\in Y^\star}\lambda_j(k,\alpha) \hbox{       \,\,\,and\,\,\,       }a_j=\min_{\alpha\in Y^\star}\lambda_j(k,\alpha)
\label{lowerupperbandedge}
\end{equation}
respectively. The band gaps are frequency intervals $\omega_{-}< \omega < \omega_+$ for which
no waves propagate inside the crystal, i.e.,
\begin{equation}
b_j<\omega_-^2<\omega_+^2<a_{j+1}.
\label{bandgapp}
\end{equation}

Over the past decades new theoretical insights into the nature of the frequency spectrum for high contrast periodic media have been made \cite{FigKuch2,FigKuch3,FigKuch1}. These efforts provide an asymptotic analysis that rigorously establishes the existence of band gaps for photonic and acoustic crystals made from thin walled cubic  lattices containing cubes of material with $a=1$ surrounded by walls with $a=k$. Band gaps are shown to appear in the limit as walls become vanishingly thin as $k\searrow 0$. 
More recently it has been shown that band gaps appear as one passes to the limit  $k\nearrow\infty$,  \cite{HempelLienau}. These gaps open in the vicinity of eigenvalues associated with the Dirichlet spectra of the included phase \cite{HempelLienau}, \cite{Friedlander}, \cite{Selden}, and \cite{AmmariKang1}. In this article we depart from previous high contrast asymptotic investigations and describe the location and  width of band gaps for finite values of the contrast $k>\overline{k}$, where $\overline{k}$ is given explicitly in terms of the crystal geometry  $\Omega$. We provide rigorous criteria that are based on the crystal geometry and material properties for opening band gaps in both 2 and 3 dimensional periodic materials, see Theorems   \ref{stopband},  \ref{passband},    \ref{stopbandgeneral}  \ref{passbandgeneral},  \ref{stopband2}, and   \ref{passband2}. These results apply  to a wide class of inclusion geometries associated with smooth boundaries.  This class of inclusion geometries include dispersions of smooth but not necessarily convex particles separated by a prescribed minimum distance, these are referred to as {\em buffered dispersions of inclusions} see section \ref{radiusgeneralshape}. The buffered dispersions are examples of a more general class of dispersions referred to as $P_\theta$ that can be characterized in a simple way in terms of energy inequalities described in Definition \ref{Ptheta}.

To illustrate the ideas we consider a photonic crystal where  $D\subset Y$  is a collection of circular  rod cross sections in the transverse plane described by $N$ disks of radius $a$.  The disks can be arranged in any configuration inside the period but neighbors can be no closer than a prescribed minimum distance $t$ inside the crystal $\Omega$ and we write $b=a+t$, see figure \ref{plane2}. We introduce the  Dirichlet spectra associated with the the Laplace operator $-\Delta$ on the inclusions. For this case the spectra is characterized by the number of disks $N$ and the Dirichlet spectrum associated with a single disk of radius $a$. We consider the part of the spectra associated with   eigenfunctions having nonzero average over the disk. These are denoted by $\delta_{0j}^\ast=\eta_{0,j}/a$ where $\eta_{0j}$ are the zeros of the Bessel function of order zero $J_0$. The Dirichlet eigenvalues 
associated with mean zero eigenfunctions are denoted by $\nu_{nk}=\eta_{nk}/a$ where $\eta_{n,k}$ is the $k^{th}$ zero of the $n^{th}$ Bessel function $J_n$, $1\leq n$. 
Next we introduce the roots $\nu_{0k}$ of the spectral function 
\begin{eqnarray}
S(\nu)=N\nu\sum_{k\in\mathbb{N}}\frac{a_{0k}^2}{\nu-(\delta_{0k}^\ast)^2}-1,
\label{extraspec}
\end{eqnarray}
where $a_{0k}=\int_D u_{0k}\,dx$ are averages of the rotationally symmetric normalized eigenfunctions $u_{0k}$ associated with the eigenvalues $\delta_{0j}^\ast$ and given by
\begin{eqnarray}
u_{0k}=J_0(r\eta_{0k}/a)/(a\sqrt{\pi}J_1(\eta_{0k})).
\label{rotintavg}
\end{eqnarray}
We write 
\begin{equation}
\sigma_N=\left\{\cup_{\scriptscriptstyle{j\in \mathbb{N}}}\nu_{0j}\right\}\bigcup\left\{\cup_{\scriptscriptstyle{(n,k)\in \mathbb{N}^2}}\nu_{nk}\right\}.
\label{neumannspectra}
\end{equation}
and the Dirichlet spectrum $\sigma(-\Delta_D)$ given by
\begin{equation}
\sigma(-\Delta_D)=\left\{\cup_{\scriptscriptstyle{j\in \mathbb{N}}}\delta_{0j}^\ast\right\}\bigcup\left\{\cup_{\scriptscriptstyle{(n,k)\in \mathbb{N}^2}}\nu_{nk}\right\}.
\label{neumannspectra}
\end{equation}

We now provide an explicit condition on the contrast $k$ that is sufficient to open a band gap in the vicinity of $\delta_{0j}^\ast$ together with explicit formulas describing its location and bandwidth. 

\begin{theorem}{Opening a band gap}\\
Given $\delta_{0j}^\ast$ define the  the set  $\sigma_N^+$ to be elements $\nu\in\sigma_N$ for which $\nu>\delta_{0j}^\ast$. The element in $\sigma_N^+$ closest to $\delta_{0j}^\ast$ is denoted by $\nu_{j+1}$. Set $d_j$ according to
\begin{equation}
d_j=\frac{1}{2}\min\left\{|\nu_{j+1}^{-1}-\nu^{-1}|;\hbox{   $\nu\in\sigma_N$}\right\}.
\label{disttospectrum}
\end{equation}
We define $\overline{r}_j$ to be
\begin{equation}
\label{upperr-j}
\overline{r}_j = \frac{\pi^2d_j(b^2-a^2)}{(b^2+a^2) + \pi^2d_j(b^2+3a^2)}.
\end{equation}
Then one has the band gap
\begin{equation}
\sigma(L_k)\cap\left(\delta_{0j}^\ast,\nu_{j+1}(1-\frac{\nu_{j+1}d_j}{k\overline{r}_j-1})\right)=\emptyset
\label{explicitgap}
\end{equation}
if
\begin{equation}
k>\overline{k}_j=\overline{r}_j^{-1}\left(1+\frac{d_j\nu_{j+1}}{1-\frac{\delta_{0j}^\ast}{\nu_{j+1}}}\right).
\label{explicitbandgap}
\end{equation}
\label{stopband}
\end{theorem}
Next we provide an explicit condition on $k$  sufficient for the persistence of a spectral band together with explicit formulas describing its location and bandwidth. 

\begin{theorem}{Persistence of passbands}\\
Given  $\delta_{0j}^\ast$ define the  the set  $\sigma_N^-$ to be elements $\nu\in\sigma_N$ for which $\nu<\delta_{0j}^\ast$. The element in $\sigma_N^-$ closest to $\delta_{0j}^\ast$ is denoted by $\nu_{j}$.  Set $d_j$ according to
\begin{equation}
d_j=\frac{1}{2}\min\left\{|(\delta^\ast_{j0})^{-1}-\delta^{-1}|;\hbox{   $\delta\in\sigma(-\Delta_D)$}\right\}.
\label{disttospectrumDirichlet}
\end{equation}
Define $\underline{r}_j$ to be
\begin{equation}
\label{lowerr-j}
\underline{r}_j = \frac{2\pi^2d_j(b^2-a^2)}{(b^2+a^2) + 2\pi^2d_j(b^2+3a^2)}.
\end{equation}
Then one has a passband in the vicinity of $\delta_{0j}^\ast$ and
\begin{equation}
\sigma(L_k)\supset\left[\nu_j,\delta_{0j}^\ast(1-\frac{\delta_{0j}^\ast d_j}{k\underline{r}_j-1})\right]
\label{explicitband}
\end{equation}
if
\begin{equation}
k>\underline{k}_j=\underline{r}_j^{-1}\left(1+\frac{d_j\delta_{0j}^\ast}{1-\frac{\nu_j}{\delta_{0j}^\ast}}\right).
\label{explicitpassband}
\end{equation}
\label{passband}
\end{theorem}

Note that both thresholds $\overline{k}_j$ and $\underline{k}_j$ depend explicitly on the crystal geometry through $a$ and $b$ and the Dirichlet spectrum of the disk. Density results on the distribution of zeros of Bessel functions \cite{Joo}, \cite{Elbert} show that the distance between adjacent eigenvalues $d_j$ approaches zero with increasing $j$. This implies together with  \eqref{upperr-j} and \eqref{explicitbandgap}  that the contrast sufficient to open gaps  grows without bound  as $j\rightarrow \infty$.   

The same mechanism can be used to open band gaps  when the coefficient satisfies $a(y)=\frac{1}{k}<1$ inside the array of inclusions and equals $1$ outside. This is shown to follow from a {\em reciprocal relation} satisfied by the spectrum. These aspects are discussed in the concluding section where a second  application to H-polarized modes inside photonic crystals is provided. 
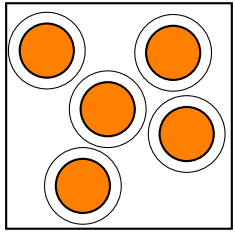
\begin{figure} 
\centering
\begin{tikzpicture}[xscale=0.6,yscale=0.6]
\draw [thick] (-2,-2) rectangle (3,3);
\draw [fill=orange,thick] (0.25,0.65) circle [radius=0.6];
\draw (0.25,0.65) circle [radius=0.85];
\draw [fill=orange,thick] (-1.1,1.95) circle [radius=0.6];
\draw (-1.1,1.95) circle [radius=0.85];
\draw [fill=orange,thick] (-0.3,-1.05) circle [radius=0.6];
\draw (-0.3,-1.05) circle [radius=0.85];
\draw [fill=orange,thick] (2.0,0.10) circle [radius=0.6];
\draw  (2.0,0.10) circle [radius=0.85];
\draw [fill=orange,thick] (1.7,1.9) circle [radius=0.6];
\draw  (1.7,1.9) circle [radius=0.85];

\end{tikzpicture} 
\caption{\bf Shaded regions are inclusions
of radius $a$ surrounded by a shell of thickness $t$.}
 \label{plane2}
\end{figure}

In this paper we introduce an approach to quantitatively describe  band structure as  illustrated in Theorems \ref{stopband} and \ref{passband}. This method is used to establish explicit conditions for  band gap opening and persistence of passbands that apply to a wide class of inclusion shapes see,  Theorem \ref{stopbandgeneral} and Theorem \ref{passbandgeneral}.  We begin by introducing the Neumann spectrum defined by the spectral problem on $Y$ given by
\begin{equation}
 -\nabla\cdot(k(1- \chi_{\scriptscriptstyle{D}}(x))+ \chi_{\scriptscriptstyle{D}}(x))\nabla h(x) =\omega^2 h(x) ,\hbox{  $x\in Y$}
\label{Eigen1neumann}
\end{equation}
\linebreak
together with the homogeneous Neumann boundary condition, where $h$ is a function in $H^1(Y)$ with $\int_Y h \,dx=0$. Here $\chi_{\scriptscriptstyle{D}}$ denotes the indicator function of $D$ in the unit period.
The Neumann eigenvalues for \eqref{Eigen1neumann} at fixed $k>1$ are written $\nu_j(k)$, $j=1,2,\ldots$ and ordered according to min-max with $0<\nu_1(k)\leq\nu_2(k),\cdots$.
The Bloch wave problem can be can also be restricted to the unit period and  is given by the spectral problem
\begin{equation}
 -\nabla\cdot(k(1- \chi_{\scriptscriptstyle{D}}(x))+ \chi_{\scriptscriptstyle{D}}(x))\nabla h(x) =\omega^2 h(x) ,\hbox{  $x\in Y$}
\label{Eigen1localBloch}
\end{equation}
\linebreak
together with the $\alpha$ quasi-periodicity condition now expressed as $h(x)=u(x)e^{i\alpha\cdot x}$, where $u(x)$ is a periodic function in $H^1_{loc}(\mathbb{R}^d)$ with unit period $Y$. For $\alpha=0$ we require $\int_Y u \,dx=0$.
For future reference we note for $k>1$ and from the min-max principle that the Neumann spectrum provides a lower bound on the Bloch spectrum given by
\begin{equation}
\nu_j(k)\leq\lambda_j(k,\alpha), \hbox{   for $\alpha\in Y^\ast$}.
\label{minmax}
\end{equation}

Our approach is based upon: 1) a representation of the Bloch eigenvalues \eqref{DispersionRelationsglobal} as convergent series expressed in terms of the contrast $k$  with explicitly defined  convergence radii as in \cite{RobertRobert1} and 2) a representation of the Neumann eigenvalues  as convergent series expressed in terms of the contrast $k$  with explicitly defined  convergence radii.
In what follows we proceed  as in \cite{RobertRobert1} to recover explicit bounds on the radii of convergence of the series representation for Neumann eigenvalues by deriving a spectral representation formula for the inverse operator $(-\nabla\cdot (k (1-\chi_{ D} )+ \chi_{D})\nabla)^{-1}$. To proceed we complexify the problem and consider $k\in\mathbb{C}$ noting that the divergence form operator $-\nabla\cdot (k(1- \chi_{D} )+ \chi_{D})\nabla$
is no longer uniformly elliptic. Our approach does not rely on ellipticity and we develop an explicit representation formula for $-\nabla\cdot (k(1- \chi_{D} ) + \chi_{D})\nabla$ that holds for complex values of $k$.  We identify the subset $z={1}/{k}\in\Omega_0$ of $\mathbb{C}$ where this operator is invertible. The explicit formula shows that the solution  operator $(-\nabla\cdot (k (1-\chi_{ D} )+ \chi_{D})\nabla)^{-1}$ is a meromorphic operator valued function of $z$ for $z\in\Omega_0=\mathbb{C}\setminus S$, see section \ref{asymptotic}  and Lemma \ref{inverseoperator}. Here the set $S$ is discrete and consists of poles lying on the negative real axis with only one accumulation point at $z=-1$. For the problem treated here we expand about $z=0$ and the distance between $z=0$ and the set $S$ is used to bound the radius of convergence for the  series. The spectral representation for $-\nabla\cdot (k  (1-\chi_{ D} ) + \chi_{D})\nabla$ follows from the existence of a complete orthonormal set of functions associated with the {\em Neumann electrostatic resonances of the crystal}, i.e., functions $v$ such that $n\cdot \nabla v=0$ on the boundary of $Y$ and real eigenvalues $\lambda$ for which
\begin{eqnarray}
\label{sourcefree}
-\nabla\cdot \chi_{D}\nabla v=-\lambda \Delta v.
\end{eqnarray}
These resonances are  connected to the spectra of Neumann-Poincar\'e operators  associated with double layer potentials discussed in \cite{Kang}, \cite{Shapero}. They are  similar in spirit to  the well known electrostatic resonances identified in the composites literature and useful for bounding effective properties \cite{BergmanES}, \cite{BergmanC}, \cite{MiltonES}, \cite{Milton}, \cite{Bruno} and \cite{GoldenPap}.   

The spectral representation is applied to analytically continue  the Neumann spectra for complex values of $k$, see Theorem \ref{extension}. 
Application of  the contour integral formula for spectral projections \cite{SzNagy}, \cite{TKato1}, \cite{TKato2} delivers an analytic representation formula for the spectral projection  associated with  the component of Neumann spectrum contained inside contours surrounding the limit eigenvalue  $\nu_j(\infty)$  see, section \ref{asymptotic}.  We provide an analytic perturbation theory  in section \ref{asymptotic} together with a calculation provided in section \ref{derivation} to find an explicit formula for the radii of convergence for the power series representation for  spectral projections. The formula shows that the radius of convergence  is determined by: 1) the distance of the origin to the nearest pole $z^\ast$ of $(-\nabla\cdot (k (1-\chi_{D}) + \chi_{D})\nabla)^{-1}$, and 2) the separation between distinct Neumann eigenvalues in the $z=1/k\rightarrow 0$ limit see Theorem \ref{separationandraduus-Neumann}. 
On restricting $k$ to be real $(-\nabla\cdot (k (1-\chi_{D}) + \chi_{D})\nabla)^{-1}$ becomes self adjoint and we recover a series representation for the Neumann eigenvalues \eqref{foureleven}, \eqref{fourtwelve}. Similar theorems on power series and radii of convergence for spectral projections associated with
the Bloch spectra $\lambda_j(k,\alpha)$  are established in \cite{RobertRobert1} making use of {\em quasi-periodic source free modes}. We adapt these results  to the present context noting that the operator is self adjoint when $k$ is real to recover a convergent series representation for Bloch eigenvalues see, \eqref{foureleven6} and \eqref{fourtwelve6}.These formulas together with an analogue of Cauchy's inequality provide explicit error estimates for the difference between the full power series and the leading order term for both Bloch and Neumann Eigenvalues  see, Theorem \ref{errorestimatesforexpansions}. We proceed to identify the leading order terms in series expansions for both Neumann and Bloch eigenvalues in section \ref{limitspeczero}. The leading order terms are shown to be elements of the limiting Dirichlet and Neumann spectrum of the included domain $D$ in agreement with   \cite{HempelLienau}. The fundamental theorems on band gaps and passbands, Theorems \ref{stopbandgeneral} and Theorem \ref{passbandgeneral}, are shown to follow from application of Theorem \ref{errorestimatesforexpansions} parts 1 and 3 and the interlacing property of the  limit spectrum as $k\rightarrow \infty$. The theory presented here can be used to quickly search for photonic and phononic band gaps within a prescribed frequency range.  Once identified these gaps can be maximized by applying topology optimization or level set methods to find inclusion shapes and lattice geometries that give the largest band gaps, see \cite{CoxDobsonH}, \cite{WangJensenSigmund2011}, \cite{MenLeeFreundPeraireJohnson}, \cite{KaoOsherYablonovich2005}, \cite{BendsoeSigmund}.

The paper is organized as follows: In the next section we introduce the Hilbert space formulation of the Neumann eigenvalue problem and the variational formulation of the Neumann electrostatic resonance problem. The completeness of the eigenfunctions associated with the electrostatic resonance spectrum is established and a spectral representation for the operator $-\nabla\cdot (k (1-\chi_{D}) + \chi_{D})\nabla$ is obtained. These results are collected and used to continue the Neumann eigenvalues  as  functions of $k$, off the real axis onto the complex plane, see  Theorem \ref{extension} of section \ref{bandstructure}. Spectral perturbation theory \cite{KatoPerturb} is applied to recover  the series representation for Neumann eigenvalues in the neighborhood of $\infty$ in section \ref{asymptotic}. The series expansion for the Bloch spectra is developed in section \ref{asymptotic6}. The leading order spectral theory for the Neumann and Bloch spectra is developed in section  \ref{limitspeczero}.  The main theorems on radius of convergence  are presented in section \ref{radius} and given by Theorems \ref{separationandraduus-alphanotzero}, \ref{separationandraduus-alphazero}, \ref{separationandraduus-Neumann}. These theorems apply to the class composite crystals $P_\theta$ described  in Definition \ref{Ptheta} provided in section \ref{radiusgeneralshape}.  The explicit radii of convergence for distributions of identical disks is presented in section \ref{radiusmultiplescatterers}. 
The fundamental theorems on band gaps and passbands are presented and proved in section \ref{bandgappassband} and are applied to recover  Theorems \ref{stopband} and  \ref{passband} for all crystal geometry's belonging to $P_\theta$. The explicit formulas for the convergence radii are derived in section \ref{derivation} as well as the proof of Theorem  \ref{separationandraduus-Neumann}  and the error estimates for  the series truncated after N terms. We conclude in section \ref{concludingsection} with a reciprocal relation and the identification of spectral gaps for the dual problem $a=\frac{1}{k}<1$ in $\Omega$ and $a=1$ outside.

\section{Hilbert space setting, Neumann electrostatic resonances and representation formulas}
\label{layers}

We denote the space of all square integrable complex valued functions $h$  defined on $Y$ with $\int_Y h\,dx=0$ by $L_{0}^2(Y)$ and the $L^2$ inner product over $Y$ is written
\begin{eqnarray}
(u,v)=\int_Y u\overline{v}\,dx.
\label{l2}
\end{eqnarray}
The eigenfunctions $h$ for  \eqref{Eigen1neumann} belong to the space
\begin{equation}
\mathcal{H} = \{ h \in  H^1(Y): \hbox{   $\int_Yh\,dx=0$} \}.
\label{H1}
\end{equation}
\linebreak
The space  $\mathcal{H}$  is a Hilbert space under the inner product
\begin{equation}
\langle u,v \rangle= \int_{Y} \nabla u(x) \cdot \nabla \bar{v}(x)dx.
\label{innerproduct}
\end{equation}
\linebreak
For any $k \in \mathbb{C}$, the weak formulation of the eigenvalue problem \eqref{Eigen1neumann} for $h$ and $\omega^2$ can be written as
\begin{eqnarray}
B_k(h,v)=\omega^2(h,v)
\label{weak}
\end{eqnarray}
for all $v$ in $\mathcal{H}$ where $B_k : \mathcal{H} \times \mathcal{H} \longrightarrow \mathbb{C}$ is the sesquilinear form
\begin{equation}
B_k(u,v) =   k\int_{Y \setminus D} \nabla u(x) \cdot \nabla \bar{v}(x)dx +  \int_{D} \nabla u(x) \cdot \nabla \bar{v}(x)dx.
\label{sesquoperator}
\end{equation}
\linebreak
The linear operator $T_k: \mathcal{H}  \longrightarrow \mathcal{H} $ associated with $B_k$ is defined by
\begin{equation}
\langle T_k u, v \rangle := B_k(u,v).
\label{Threetwo}
\end{equation}

In what follows we decompose $\mathcal{H}$  into invariant subspaces of source free modes and identify the associated Neumann electrostatic resonance spectra. This decomposition will provide an explicit spectral representation for the operator $T_k$, see Theorem \ref{T_z spectrum}. Let $W_1 \subset \mathcal{H}$ be the completion in $\mathcal{H}$ of the subspace of functions with support away from $D$. Now let $\tilde{H}^1_0(D)$ denote the subspace of functions $H^1_0(D)$ extended by zero into $Y\setminus D$ and let $1_Y$ be the indicator function of $Y$. We define $W_2 \subset \mathcal{H}$ to be the subspace of functions given by 
\begin{equation}
\label{definitionW2periodic}
W_2 = \{ u = \tilde{u} - \left( \int_D \tilde{u} dx \right) 1_Y \; \mid \;  \tilde{u} \in \tilde{H}^1_0(D)\}.
\end{equation}
Clearly $W_1$ and $W_2$ are orthogonal subspaces of $\mathcal{H}$, so define $W_3 := (W_1\oplus W_2)^{\bot}$.  We therefore have 
\begin{equation}
\mathcal{H}= W_1 \oplus W_2 \oplus W_3.
\end{equation}
The orthogonal decomposition and integration by parts shows that elements $u \in W_3$ are harmonic separately in $D$ and $Y\setminus D$, and 
\begin{eqnarray}
\label{normal1}
&&\frac{\partial u}{\partial n} =0\hbox{  on  $\partial Y$     and  }\\
&&\int_{\partial D}{\frac{\partial u}{\partial n}}\mid^{+}_{\partial D}\, ds=0,\label{normal2}\\
&&\int_{\partial D}{\frac{\partial u}{\partial n}}\mid^{-}_{\partial D}\, ds=0,\label{normal3}
\end{eqnarray}
where ${\frac{\partial u}{\partial n}}\mid^{-}_{\partial D}$ and $ {\frac{\partial u}{\partial n}}\mid^{+}_{\partial D}$ are traces of the outward directed normal derivative taken from the interior of $D$ and exterior of $D$ respectively.

To set up the spectral analysis note  that elements of $W_3$ can be represented in terms of single layer potentials supported on $\partial D$.  We introduce the  $d$-dimensional Newtonian potential, $d=2,3$
given by
\begin{equation}
\Gamma_d(x,y) = \left\{\frac{1}{2\pi}ln{|x-y|} \hbox{  for   $d=2$,    and   }-\frac{1}{4\pi}|x-y|^{-1}\hbox{   for  $d=3$  }\right\}.
\label{Newtonian}
\end{equation}
Let $\phi(x,y)$ satisfy
\begin{equation}
 -\Delta_y \phi(x,y) =0, \hbox{    for   $(x,y)\in Y\times Y$},
\label{Homogeneous}
\end{equation}
with
\begin{equation}
\label{boundarycorrection}
\frac{\partial \phi(x,y)}{\partial n(y)}=-\frac{\partial \Gamma_d(x,y)}{\partial n(y)} \hbox{   for   
   $y\in\partial Y$       and    $x\in Y$}.
\end{equation}
The Neumann Green's function  is given by
\begin{equation}
G(x,y)=\Gamma_d(x,y)+\phi(x,y).
\label{neumanngreens}
\end{equation}
Let $H^{1/2}(\partial D)$ be the fractional Sobelev space on $\partial D$ defined in the usual way, and its dual by $H^{-1/2}(\partial D)$.  For $\rho \in H^{-1/2}(\partial D)$  the single layer potential is given by
\begin{equation}
S_D\rho(x) = \int_{\partial D} G(x,y)\rho(y)d\sigma(y)\text{,           }x\in Y.
\end{equation}
The jump in normal derivative across $\partial D$ belongs to $H^{-1/2}(\partial D)$ and is written:
\begin{equation}
\rho_u={\frac{\partial u}{\partial n}}\mid^{+}_{\partial D}-{\frac{\partial u}{\partial n}}\mid^{-}_{\partial D}
\label{jumpnormal}
\end{equation}
For $u\in W_3$ we have the identity
\begin{equation}
u=S_D\rho_u
\label{identforW3}
\end{equation}
where $\int_{\partial D}\rho_u\,ds=0$ follows from \eqref{normal2} and \eqref{normal3}.
and we introduce $H_0^{-1/2}(\partial D)=\{\rho\in H^{-1/2}(\partial D):\,\int_{\partial D}\rho\,ds=0\}$, to see that $S_D:\, H_0^{-1/2}(\partial D) \rightarrow W_3$ maps $H_0^{-1/2}(\partial D)$ onto $W^3$.
It follows from \cite{CostabelBdryOps}, that for any $\rho \in H^{-1/2}(\partial D)$ 
\begin{eqnarray}
\Delta S_D\rho & = & 0\text{ in } D \text{ and } Y\setminus D,\nonumber \\
S_D\rho \mid^{+}_{\partial D} & = & S_D\rho \mid^{-}_{\partial D},\nonumber\\
\frac{\partial}{\partial n}S_D\rho \mid^{\pm}_{\partial D} & = & \pm \frac{1}{2}\rho + {K}_D^{*}\rho,
\label{singlelayer}
\end{eqnarray}
where $n$ is the outward directed normal vector on $\partial D$ and ${K}_D^{*}$ is the Neumann Poincar\'e operator defined by
\begin{equation}
{K}_D^{\ast}\rho(x) = \text{ p. v. } \int_{\partial D} \frac{\partial G(x,y)}{\partial n(x)}\rho(y)d\sigma(y)\text{,    } x \in \partial D,
\label{neumannp}
\end{equation}
and
$K_D$ is the Neumann Poincar\'e operator
\begin{equation}
K_D\rho(x) = \text{ p. v. } \int_{\partial D} \frac{\partial G(x,y)}{\partial n(y)}\rho(y)d\sigma(y)\text{,    } x \in \partial D.
\label{adjointneumannp}
\end{equation}

In what follows we assume the boundary  $\partial D$ is $C^{1,\gamma}$, for some $\gamma>0$.
Here the layer potentials  ${K}_D$, and ${K}_D^*$ are continuous linear mappings from $L^2(\partial D)$ to $L^2(\partial D)$ and compact, since  $\frac{\partial G(x,y)}{\partial n(x)}$ is a continuous kernel of order $d-2$ in dimensions $d=2,3$.  The operator $S_{D}$ is a continuous linear map from $H^{-1/2}_0(\partial D)$ into $W_3$ and we define $S_{\partial D} \rho = S_D \rho \mid_{\partial D}$ for all $\rho \in H^{-1/2}_0(\partial D)$.  Here $S_{\partial D} : H^{-1/2}_0(\partial D) \longrightarrow H^{1/2}(\partial D)$ is continuous and has bounded inverse, see \cite{CostabelBdryOps}.

One readily verifies the symmetry 
\begin{eqnarray}
G(x,y)=G(y,x),
\label{greensymmetries}
\end{eqnarray}
and application delivers the Plemelj symmetry for $K_D$, ${K}^*_D$ and $S_{\partial D}$ as operators on $L^2(\partial D)$  given by
\begin{eqnarray}
K_DS_{\partial D}=S_{\partial D}K^*_D.
\label{Plemelj}
\end{eqnarray}
Moreover as seen in \cite{Shapero} the operator $-S_{\partial D}$ is positive and selfadjoint in $L^2(\partial D)$ and in view of \eqref {Plemelj}    $K_D^*$ is  a compact operator on  $H_0^{-1/2}(\partial D)$.

Let $G : W_3 \longrightarrow H^{1/2}(\partial D)$ be the trace operator, which is bounded and onto.
\begin{lemma}
$S_D: H_0^{-1/2}(\partial D) \longrightarrow W_3$ has bounded inverse $S_D^{-1} = S_{\partial D}^{-1}G$.
\end{lemma}
\begin{proof}
Suppose $u\in W_3$, and consider $G u = u\mid_{\partial D} \in H^{1/2}(\partial D)$.  For all $x \in Y$ define $w(x) = S_D(S_{\partial D}^{-1}G u)$.  Since $u, w \in W_3$, it follows that $w-u \in W_3$ as well.  Since $G u = G w$, we have that $G (w-u) = 0$, and so $w-u \in (W_1 \oplus W_2)$. But $W_3 = (W_1\oplus W_2)^{\bot}$, so $w=u$ as desired. The boundedness follows from the continuity of $S_{\partial D}^{-1}$ and $G$.
\end{proof}
We introduce an auxiliary operator $T:W_3 \longrightarrow W_3$, given by the sesquilinear form
\begin{equation}
\langle Tu,v \rangle = \frac{1}{2} \int_{Y \setminus D} \nabla u(x) \cdot \nabla \bar{v}(x)dx - \frac{1}{2}\int_{D} \nabla u(x) \cdot \nabla \bar{v}(x)dx.
\label{threefourteen}
\end{equation}
The next theorem will be useful for the spectral decomposition of $T_k$.
\begin{theorem}
The linear map $T$ defined in equation \eqref{threefourteen} is given by
$$T = S_D K_D^*S_D^{-1}$$
and is compact and self-adjoint.
\end{theorem}
\begin{proof}
For $u,v \in W_3$, consider
\begin{equation}
\langle S_D K_D^*S_D^{-1}u, v \rangle = \int_Y \nabla [S_D K_D^*S_D^{-1}u] \cdot \nabla \bar{v}.
\end{equation}
Since $\Delta S_D\rho = 0$ in $D$ and $Y\setminus D$ for any $\rho \in H_0^{-1/2}(\partial D)$, an integration by parts yields
$$\langle S_D K_D^*S_D^{-1}u, v \rangle = \int_{\partial D} \bar{v} ( \frac{\partial [S_D {K}_D^*S_D^{-1}u]}{\partial \nu}\mid^-_{\partial D} -  \frac{\partial [S_D K_D^*S_D^{-1}u]}{\partial \nu}\mid^+_{\partial D})d\sigma.$$
Applying the jump conditions from \eqref{singlelayer} yields
\begin{equation}
\langle S_D K_D^*S_D^{-1}u, v \rangle = -\int_{\partial D} K_D^*S_D^{-1}u\bar{v}d\sigma.
\label{Tboundary}
\end{equation}
Note that by the same jump conditions
\begin{equation}
\label{samejump}
K_D^*S_D^{-1}u = \frac{1}{2}(\frac{\partial u}{\partial \nu}|^-_{\partial D} + \frac{\partial u}{\partial \nu}|^+_{\partial D}).
\end{equation}
Application of \eqref{samejump} to equation \eqref{Tboundary} and an integration by parts yields the desired result. Compactness follows directly from the properties of 
$S_D$ and $K^\ast$.
\end{proof}

Rearranging terms in the weak formulation of \eqref{sourcefree} and writing $\mu=1/2-\lambda$ delivers the equivalent eigenvalue problem for Neumann electrostatic resonances:

$$\langle Tu,v \rangle = \mu\langle u,v \rangle\text{,  } u,v \in W_3.$$
Since $T$ is compact and self adjoint on $W_3$, there exists a countable subset $\{ \mu_i \}_{i\in \mathbb{N}}$ of the real line with a single accumulation point at $0$ and an associated family of orthogonal finite-dimensional projections $\{ P_{\mu_i}\}_{i\in \mathbb{N}}$ such that
$$\langle \sum_{i=1}^\infty P_{\mu_i}u, v \rangle = \langle u,v \rangle \text{,       } u,v \in W_3$$
and
$$\langle \sum_{i=1}^\infty \mu_iP_{\mu_i}u, v \rangle = \langle Tu,v \rangle \text{,       } u,v \in W_3.$$
Moreover, it is clear by \eqref{threefourteen} that 
$$-\frac{1}{2} \leq \mu_i \leq \frac{1}{2}.$$
The upper bound $1/2$ is the eigenvalue associated with the eigenfunction $\Pi\in W_3$ such that $\Pi=1$ in $D$.  It is also easy to show that there are no nonzero elements of $W_3$ that are eigenfunctions associated with eigenvalue $\mu=-1/2$.  In section \ref{radiusgeneralshape} an explicit lower bound $\mu^-$ is identified such that the inequality $-1/2<\mu^-\leq \mu_i$,   holds uniformly for a very broad class of geometries.
\begin{lemma}
The eigenvalues $\{\mu_i\}_{i\in\mathbb{N}}$ of $T$ are precisely the eigenvalues of the Neumann-Poincar\'e operator  ${K}^*$ associated with quasi-periodic double layer potential restricted to $\partial D$. \end{lemma}
\begin{proof}
If a pair $(\mu, u)$ belonging to $(-1/2, 1/2] \times W_3$ satisfies $Tu = \mu u$ then
$S_D K_D^\ast S_D^{-1}u= \mu u$. Multiplication of both sides by $S_D^{-1}$ shows that $S_D^{-1}u$ is an eigenfunction for  
function for $K_D^\ast$  associated with $\mu$. Suppose the pair $(\mu,w)$ belongs to $(-1/2,1/2]\times
H^{-1/2}(\partial D)$ and satisfies $K_D^\ast w=\mu w$. Since the trace map from $W_3$ to $H_0^{1/2}(\partial D)$ is onto then there is a  $u$ in $W_3$ for which $w=S_D^{-1}u$ and  $K_D^\ast S_D^{-1}u=\mu S_D^{-1}u.$
Multiplication of this identity by $S_D$ shows that $u$ is an eigenfunction for $T$ associated with $\mu$.
\end{proof}

Finally, we see that if $u_1\in W_1$ and $u_2\in W_2$, then  

$$\langle Tu_1,v \rangle = \frac{1}{2}\langle u_1,v \rangle \text{,}$$
$$\langle Tu_2, v\rangle = -\frac{1}{2}\langle u_2,v\rangle $$
for all $v\in H^1_{\#}(\alpha,Y)$.

Let $Q_1, Q_2$ be the orthogonal projections of $\mathcal{H}$ onto $W_1$ and $W_2$ respectively, and define $P_1 := Q_1 + P_{1/2} \text{,  } P_2:= Q_2 $.  Here $P_{1/2}$ is the projection onto the one dimensional subspace spanned by the function $\Pi\in W_3$. Then $\{ P_1, P_2\} \cup \{ P_{\mu_i}\}_{-\frac{1}{2} < \mu_i < \frac{1}{2}}$ is an orthogonal family of projections, and
$$\langle P_1u + P_2u + \sum_{-\frac{1}{2} < \mu_i < \frac{1}{2}} P_{\mu_i}u, v\rangle = \langle u,v\rangle$$
for all $u,v \in \mathcal{H}$.

We now recover the spectral decomposition for $T_k$ associated with the sesqualinear form \eqref{Threetwo}.
\begin{theorem}
\label{T_z spectrum}
The linear operator $T_k: \mathcal{H}\longrightarrow \mathcal{H}$ associated with the sesqualinear form $B_k$ is  is given by
$$\langle T_k u,v \rangle = \langle k P_1u + P_2u + \sum_{-\frac{1}{2} < \mu_i < \frac{1}{2}}[k(1/2 + \mu_i) + (1/2-\mu_i)] P_{\mu_i}u, v\rangle$$
for all $u,v\in \mathcal{H}$.
\end{theorem}
\begin{proof}
For $u, v\in \mathcal{H}$ we have
$$B_k(P_{\mu_i}u, v) = k\int_{Y \setminus D} \nabla P_{\mu_i}u \cdot \nabla \bar{v} +  \int_{D} \nabla P_{\mu_i}u \cdot \nabla \bar{v}.$$
Since $P_{\mu_i}u$ is an eigenvector corresponding to $\mu_i \neq \pm \frac{1}{2}$, we have
$$\int_{Y \setminus D} \nabla P_{\mu_i}u \cdot \nabla \bar{v} = \frac{(1/2+\mu_i)}{(1/2-\mu_i)}\int_{D} \nabla P_{\mu_i}u \cdot \nabla \bar{v}$$
and so we calculate
$$B_k(P_{\mu_i}u, v) = [k\frac{(1/2+\mu_i)}{(1/2-\mu_i)} +1]\int_{D} \nabla P_{\mu_i}u \cdot \nabla \bar{v}.$$
But we also know that
$$\int_D \nabla P_{\mu_i} u \cdot \nabla \bar{v} = (1/2-\mu_i)\int_Y \nabla P_{\mu_i} u \cdot \nabla \bar{v}$$
and so
$$B_k(P_{\mu_i}u, v) = [k(1/2 + \mu_i) + (1/2-\mu_i)] \int _Y \nabla P_{\mu_i} u \cdot \nabla \bar{v}.$$
Since we clearly have
$$B_k(P_1u, v) = k \int_{Y \setminus D} \nabla P_1u \cdot \nabla \bar{v}\text{,}$$
$$B_k(P_2u, v) = \int_{D} \nabla P_2u \cdot \nabla \bar{v}\text{,}$$
and the projections $P_1, P_2, P_{\mu_i}$ are mutually orthogonal for all $-\frac{1}{2} < \mu_i < \frac{1}{2}$, the proof is complete.
\end{proof}

It is evident that $T_k$ is invertible whenever 
\begin{eqnarray} k \in \mathbb{C}\setminus Z \hbox{ where } Z=\{ \frac{\mu_i - 1/2}{\mu_i + 1/2} \}_{\{-\frac{1}{2} \leq \mu_i \leq \frac{1}{2}\}}
\label{invertability}
\end{eqnarray}
 and for $z=k^{-1}$,
\begin{eqnarray}
(T_k)^{-1}= z P_1u + P_2u + \sum_{-\frac{1}{2} < \mu_i < \frac{1}{2}}z[(1/2 + \mu_i) + z(1/2-\mu_i)] ^{-1}P_{\mu_i}.
\label{inverse}
\end{eqnarray}
For future reference we also introduce the set $S$ of $z\in\mathbb{C}$ for which $T_k$ is not invertible given by
\begin{eqnarray} 
S=\{ \frac{\mu_i + 1/2}{\mu_i - 1/2} \}_{\{-\frac{1}{2} < \mu_i < \frac{1}{2}\}}
\label{noninvertability}
\end{eqnarray}
which also lies on the negative real axis. In section \ref{radiusgeneralshape} we will provide explicit upper bounds on the set $S$ that depend upon the geometry of the inclusions.

\section{Neumann spectrum for Complex Coupling Constant }
\label{bandstructure}

We set $\omega^2=\nu(k)$ in  \eqref{Eigen1neumann} and extend the Neumann eigenvalue problem to complex coefficients $k$ outside the set $Z$ given by \eqref{invertability}. The Neumann spectral problem \eqref{Eigen1neumann}is written as

\begin{eqnarray}
 -\nabla\cdot (k (1-\chi_{D}) + \chi_{D})\nabla u=-\Delta_N T_k u=\nu(k) u,
 \label{representationform}
 \end{eqnarray}
with $u\in\mathcal{H}$. Here $-\Delta_N$ is the Laplace operator associated with the bilinear form $\langle\cdot,\cdot\rangle$ defined on $\mathcal{H}$.
We characterize the Bloch spectra by analyzing the operator 
\begin{eqnarray}
B(k)=T_k^{-1}(-\Delta_N^{-1}),
\label{inverseB}
\end{eqnarray}
where the operator $-\Delta_N^{-1}$ exists by the Lax Milgram Lemma and has an explicit spectral representation in terms of the Neumann eigenfunctions for the Laplacian on the unit period cell.

The operator $B(k): L^2_{0}(Y) \longrightarrow \mathcal{H}$ is easily seen to be bounded for $k\not\in Z$, see Theorem \ref{bounded}. Since $\mathcal{H}\subset H^1(Y)$ embeds compactly into $L^2_{0}(Y)$ we find by virtue of Poincare's inequality that $B(k)$ is a bounded compact linear operator on $L^2_0(Y)$ and therefore has a discrete spectrum $\{ \gamma_i(k) \}_{i \in \mathbb{N}}$ with a possible accumulation point at $0$, see Remark \ref{compact2}. The corresponding  eigenspaces are finite dimensional and the eigenfunctions $p_i\in L^2_{0}(Y)$ satisfy
\begin{eqnarray}
B(k)p_i(x)=\gamma_i(k,\alpha)p_i(x)\hbox{ for $x$ in $Y$}
\label{compactproblem}
\end{eqnarray}
and also belong to $\mathcal{H}$. Note further for $\gamma_i\not =0$ that \eqref{compactproblem} holds if and only if \eqref{representationform} holds with $\nu_i(k)=\gamma_i^{-1}(k)$, and $-\Delta_N T_k u_i=\nu_i(k) u_i.$
Collecting results we have the following theorem
\begin{theorem}
\label{extension}
Let $Z$ denote the set of points on the negative real axis defined by \eqref{invertability}. Then the Neumann eigenvalue problem \eqref{Eigen1neumann}  can be extended for values of the coupling constant $k$ off the positive real axis into $\mathbb{C}\setminus Z$, i.e.,  the Neumann eigenvalues are of finite multiplicity and denoted by $\nu_j(k)=\gamma_j^{-1}(k)$, $j\in \mathbb{N}$.
\end{theorem}

\section{Series Representation of Neumann Eigenvalues}
\label{asymptotic}
In what follows we set $\gamma=\nu^{-1}(k)$ and analyze the spectral problem
\begin{equation}
B(k) u = \gamma (k) u
\label{forpointtwo}
\end{equation}
Henceforth we will analyze the high contrast limit by  by developing a power series in $z=\frac{1}{k}$ about $z=0$ for the spectrum of the family of operators associated with \eqref{forpointtwo}. 
$$\begin{array}{lcl}
B(k) & := & T_k^{-1}(-\Delta_{N})^{-1}\\
& = &  (zP_1 + P_2 + z\sum_{-\frac{1}{2} < \mu_i < \frac{1}{2}}[(1/2 + \mu_i) + z(1/2-\mu_i)]^{-1} P_{\mu_i})(-\Delta_{N})^{-1}\\
&=& A(z).
\end{array}$$
Here we define the operator $A(z)$ such that $A(1/k)=B(k)$ and the associated eigenvalues $\beta_j(1/k)=\gamma_j(k)$ and the spectral problem is $A(z)u=\beta(z)u$ for $u\in L^2_{0}(Y)$. The Neumann eigenvalues are of finite multiplicity and described in terms  of $\beta_j(z)\in\sigma(A(z))$ by 
\begin{equation}
\nu_j(k)=\frac{1}{\beta_j(\frac{1}{k})}, \hbox{   $j\in \mathbb{N}.$}
\label{relationbeteigenvals}
\end{equation}

It is readily seen from the above representation that  $A^{\alpha}(z)$ is self-adjoint for $k\in\mathbb{R}$ and is a family of bounded operators taking $L^2_{0}(Y)$ into itself see, Theorem \ref{Selfadjoint} and subsequent remarks.  We have the following:
\begin{lemma}
\label{inverseoperator}
$A(z)$ is holomorphic on $\Omega_0 := \mathbb{C} \setminus S$. 
Where $S=\cup_{i\in\mathbb{N}} z_i$ is the collection of points $z_i=(\mu_i+1/2)/(\mu_i-1/2)$ on the negative real axis associated with
the eigenvalues $\{\mu_i\}_{i\in\mathbb{N}}$. The set $S$ consists of poles  of $A(z)$ with only one accumulation point at $z=-1$.
\end{lemma}

In the sections \ref{radiusgeneralshape}   and \ref{radiusmultiplescatterers} we identify explicit lower bounds  $-1/2<\mu^\ast\leq \min_i\{\mu_i\}$, that hold for generic classes of inclusion domains $D$. The corresponding upper bound $z^\ast$ on $S$ is written
\begin{eqnarray}
 \max_i \{z_i\}\leq\frac{\mu^\ast+1/2}{\mu^\ast-1/2}= z^\ast<0.
\label{bdsonS}
\end{eqnarray}

Let $\beta_0 \in \sigma(A(0))$ with spectral projection $P(0)$, and let $\Gamma$ be a closed contour in $\mathbb{C}$ enclosing $\beta_0$ but no other $\beta_0 \in \sigma(A(0))$.
The spectral projection associated with $\beta(z) \in \sigma(A(z))$ for $\beta(z) \in \text{int}(\Gamma)$ is denoted by $P(z)$. Here we suppose all elements $\beta(z) \in \text{int}(\Gamma)$ converge to $\beta_0$ as $z\rightarrow0$. This collection is known as the eigenvalue group associated with $\beta_0$. We write $M(z) = P(z)L^2_{0}(Y)$ and suppose for the moment that $\Gamma$ lies in the resolvent of $A(z)$ and $\text{dim}(M(0)) =\text{dim}(M(z))= m$, noting that  Theorems \ref{separationandraduus-alphanotzero}, \ref{separationandraduus-alphazero} and \ref{separationandraduus-Neumann}  provide explicit conditions for when this holds true.  
Since $A(z)$ is analytic in a neighborhood of the origin we  write
\begin{equation}
A(z) = A(0) + \sum_{n=1}^{\infty} z^nA_n.
\end{equation}
Define the resolvent of $A(z)$ by
$$R(\zeta, z) = (A(z) - \zeta )^{-1}\text{,}$$
and expanding successively in Neumann series and power series we have the identity

\begin{equation}
\begin{array}{lcl}
R(\zeta,z) & = & R(\zeta,0)[I + (A(z) - A(0))R(\zeta,0)]^{-1} \\
\\
& = & R(\zeta,0)+\sum_{p=1}^\infty [-(A(z) - A(0))R(\zeta,0)]^p\\
\\
& = & R(\zeta,0) + \sum_{n=1}^{\infty} z^nR_n(\zeta)\text{,}
\end{array}
\label{foursix}
\end{equation}

where
$$R_n(\zeta) = \sum_{k_1 + \ldots k_p = n, k_j \geq 1} (-1)^pR(\zeta,0)A_{k_1}R(\zeta,0)A_{k_2}\ldots R(\zeta,0)A_{k_p}$$
for $n\geq 1$.

Application of  the contour integral formula for spectral projections \cite{SzNagy}, \cite{TKato1}, \cite{TKato2} delivers the expansion for the spectral projection
\begin{equation}
\begin{array}{lcl}
P(z) & = & -\frac{1}{2\pi i} \oint_{\Gamma} R(\zeta, z)d\zeta\\
\\
& = & P(0) + \sum_{n=1}^\infty z^n P_n
\end{array}
\label{Project1}
\end{equation}
where $P_n =  -\frac{1}{2\pi i} \oint_{\Gamma} R_n(\zeta)d\zeta$.  Now we represent  the  difference $(A(z) - \beta_0)P(z)$ as power series. Start with
\begin{eqnarray}
(A(z) - \beta_0)R(\zeta,z) = I + (\zeta - \beta_0)R(\zeta,z)
\label{project2}
\end{eqnarray}
and  we have 
\begin{equation}
(A(z) - \beta_0)P(z) = - \frac{1}{2 \pi i}\oint_{\Gamma} (\zeta - \beta_0)R(\zeta,z)d\zeta \text{,}
\label{cauchyformula}
\end{equation}
and the power series representation for the operator $A(z)$ follows from \eqref{foursix}.

Next we develop a series representation for the eigenvalues of $A(z)$ for real $z=1/k$. We observe that the operator $A(z)$ is selfadjoint and bounded for $z\in\mathbb{R}\not\in S$ see, Theorem \ref{Selfadjoint}.
So for $z\in \mathbb{R}\setminus S$ there is a complete orthonormal system of eigenfunctions $\{\varphi_i(z)\}_{i\in\mathbb{N}}$ in $L^2_0(Y)$ and eigenvalues $\{\beta_i(z)\}_{i\in\mathbb{N}}$ such that 
\begin{equation}
\label{complete}
\beta_i(z)=(A(z)\varphi_i(z),\varphi_i(z)), \hbox{   for   $z\in \mathbb{R}\setminus S$}.
\end{equation}
Now for $\beta_0$ such that $\beta_0=(A(0)\varphi_i(0),\varphi_i(0))$ and $\beta_i(z)$ associated with the eigenvalue group corresponding to $\beta_0$ inside $\Gamma$ we have $P(z)\varphi_i(z)=\varphi_i(z)$ and from \eqref{cauchyformula}
\begin{eqnarray}
&&\beta_i(z) - \beta_0 =((A(z) - \beta_0)P(z)\varphi_i(z),\varphi_i(z))\nonumber\\
&&=- (\frac{1}{2 \pi i}\oint_{\Gamma} (\zeta - \beta_0)R(\zeta,z)d\zeta\varphi_i(z),\varphi_i(z)).
\label{fourten}
\end{eqnarray}
The eigenfunctions $\varphi_i(z)$ are analytic in $z$ see (\cite{KatoPerturb}, Chapter II \S 4 \& Chapter VII \S 3).

We develop a series for the eigenvalues suitable for the analysis. We apply \eqref{inverse} and write
\begin{equation}
\begin{array}{lcl}
R(\zeta,z) & = & R(\zeta,0)+\sum_{p=1}^\infty [-(A(z) - A(0))R(\zeta,0)]^p\\
\\
& = & R(\zeta,0) + \sum_{n=1}^{\infty} z^n{\mathcal{N}}^n(\zeta,z)\text{,}
\end{array}
\label{foursixalternate}
\end{equation}
where
\begin{equation}
\label{caligraphN}
\mathcal{N}(\zeta,z) =  [(P_1u + \sum_{-\frac{1}{2} < \mu_i < \frac{1}{2}}[(1/2 + \mu_i) + z(1/2-\mu_i)] ^{-1}P_{\mu_i})(-\Delta_N)^{-1}R(\zeta,0)]
\end{equation}
for $n\geq 1$.
Equation \eqref{fourten} together with \ref{foursixalternate} deliver a series representation formula for  Neumann eigenvalues $\nu_i(k)=(\beta_i(\frac{1}{k}))^{-1}$ for real $k$ in a neighborhood of $\infty$. Substituting \eqref{foursixalternate} into \eqref{fourten} and manipulation  yields
\begin{equation}
{\beta}_i(z) = \beta_0 + \sum_{n=1}^\infty z^n\beta^n_i(z),
\label{foureleven}
\end{equation}
where
\begin{equation}
{\beta}^n_i(z) = - \frac{1}{2\pi i} \left(\oint_{\Gamma}(\zeta-\beta_0)\mathcal{N}^n(\zeta,z)d\zeta\varphi_i(z),\varphi_i(z)\right);\hbox{ $n\geq 1$}.
\label{fourtwelve}
\end{equation}
for $z$ in an interval containing $z=0$.

\section{Series Representation of Bloch Eigenvalues}
\label{asymptotic6}

We begin with the variational description of the Bloch eigenvalue problem \eqref{Eigen1localBloch}.
Denote the spaces of all $\alpha$ quasi-periodic complex valued functions belonging to $L_{loc}^2(\mathbb{R}^d)$ by $L_{\#}^2(\alpha,Y)$. For $\alpha\not=0$ the eigenfunctions $h$ for  \eqref{Eigen1localBloch} belong to the space
\begin{equation}
H^1_{\#}(\alpha,Y) = \{ h \in  H_{loc}^1(\mathbb{R}^d): \hbox{$h$ is $\alpha$ quasiperiodic} \}.
\label{H1}
\end{equation}
\linebreak
The space  $H^1_{\#}(\alpha,Y)$  is a Hilbert space under the inner product \eqref{innerproduct}.
When $\alpha=0$, the pair $h(x)=1$, $\omega^2=0$ is a solution to \eqref{Eigen1localBloch}. For this case the remaining eigenfunctions associated with nonzero eigenvalues are orthogonal to  $1$ in the $L^2(Y)$ inner product.  These eigenfunctions belong to the  set of square integrable $Y$ periodic functions with zero average denoted by $L^2_{\#}(0,Y)$.  They also belong to the space
\begin{equation}
H^1_{\#}(0,Y) = \{ h \in  H_{loc}^1(\mathbb{R}^d): \hbox{$h$ is periodic, $\int_Y h\,dx=0$} \}.
\label{H1}
\end{equation}
The space  $H^1_{\#}(0,Y)$  is also Hilbert space with  inner product defined by \eqref{innerproduct}.

For any $k \in \mathbb{C}$, the the variational formulation of the Bloch eigenvalue problem \eqref{Eigen1localBloch} for $h$ and $\omega^2$ is given by
\begin{eqnarray}
B_k(h,v)=\omega^2(h,v)
\label{weak}
\end{eqnarray}
for all $v$ in $H^1_{\#}(\alpha,Y)$ and $\alpha\in Y^\ast$. As before it is possible to decompose $H^1_{\#}(\alpha,Y)$ into an orthogonal sum of three subspaces $W_1$, $W_2$, and $W_3$ to recover an analytic representation for the operator $T_k^\alpha$ associated with the bilinear form $B_k(h,v)$ defined on $H^1_{\#}(\alpha,Y)\times H^1_{\#}(\alpha,Y)$.

We first address the case $\alpha\in Y^\star\setminus \{0\}$. Let $W_1 \subset H^1_{\#}(\alpha,Y)$ be the completion in $H^1_{\#}(\alpha,Y)$ of the subspace of functions with support away from $D$, and let $W_2 \subset H^1_{\#}(\alpha,Y)$ be the subspace of functions in $H^1_0(D)$ extended by zero into $Y$.  Clearly $W_1$ and $W_2$ are orthogonal subspaces of $H^1_{\#}(\alpha,Y)$, so define $W_3 := (W_1\oplus W_2)^{\bot}$.  We therefore have 
\begin{equation}
H^1_{\#}(\alpha,Y) = W_1 \oplus W_2 \oplus W_3.
\end{equation}
The orthogonal decomposition and integration by parts shows that elements $u \in W_3$ are harmonic separately in $D$ and $Y\setminus D$.

Now consider $\alpha=0$ and decompose $H^1_\#(0,Y)$. Let $W_1 \subset H^1_{\#}(0,Y)$ be the completion in $H^1_{\#}(0,Y)$ of the subspace of functions with support away from $D$. Here let $\tilde{H}^1_0(D)$ denote the subspace of functions $H^1_0(D)$ extended by zero into $Y\setminus D$ and let $1_Y$ be the indicator function of $Y$. We define $W_2 \subset H^1_{\#}(0,Y)$ to be the subspace of functions given by 
\begin{equation}
W_2 = \{ u = \tilde{u} - \left( \frac{1}{|Y|} \int_D \tilde{u} dx \right) 1_Y \; \mid \;  \tilde{u} \in \tilde{H}^1_0(D)\}
\end{equation}
Clearly $W_1$ and $W_2$ are orthogonal subspaces of $H^1_{\#}(0,Y)$, and $W_3 := (W_1\oplus W_2)^{\bot}$.  As before we have 
\begin{equation}
H^1_{\#}(0,Y) = W_1 \oplus W_2 \oplus W_3
\end{equation}
and $W_3$ is identified with the subspace of $H^1_{\#}(0,Y)$ functions that are harmonic inside $D$ and $Y\setminus D$ respectively. The orthogonality between  $W_2$ and $W_3$ follows from the identity $\int_{\partial D}\partial_n w \, ds=0$ for $w\in W_3$.

For every $\alpha\in Y^\ast$ the associated  subspace $W_3$ can be decomposed into finite dimensional pairwise orthogonal subspaces associated with the eigenvalue problem for quasi-periodic electrostatic resonances.
$$\langle Tu,v \rangle = \mu\langle u,v \rangle\text{,  } u,v \in W_3$$
The operator $T$ defined by the sesqualinear form \eqref{threefourteen} is seen to be compact and selfadjoint on $W_3$ see, \cite{RobertRobert1}.
Since $T$ is compact and self adjoint on $W_3$, there exists a countable subset $\{ \mu_i \}_{i\in \mathbb{N}}$ of the real line with a single accumulation point at $0$ and an associated family of orthogonal finite-dimensional projections $\{ P_{\mu_i}\}_{i\in \mathbb{N}}$ such that
$$\langle \sum_{i=1}^\infty P_{\mu_i}u, v \rangle = \langle u,v \rangle \text{,       } u,v \in W_3$$
and
$$\langle \sum_{i=1}^\infty \mu_iP_{\mu_i}u, v \rangle = \langle Tu,v \rangle \text{,       } u,v \in W_3$$
where
$$-\frac{1}{2} \leq \mu_i \leq \frac{1}{2},$$
see \cite{RobertRobert1}.
The upper bound $1/2$ is the eigenvalue associated with the eigenfunction $\Pi\in H_{\#}^1(\alpha,Y)$ such that $\Pi=1$ in $D$ and is harmonic on $Y\setminus D$.  In \cite{RobertRobert1} an explicit lower bound $\mu^*$ is identified such that the inequality $-1/2<\mu^*\leq \mu_i$,   holds for a generic class of geometries uniformly with respect to $\alpha\in Y^\star$. It is shown that this bound is independent of $\alpha\in Y^\ast$ see \cite{RobertRobert1} and is discussed in sections \ref{radiusgeneralshape}
and \ref{radiusmultiplescatterers}.

Finally, we see that if $u_1\in W_1$ and $u_2\in W_2$, then  

$$\langle Tu_1,v \rangle = \frac{1}{2}\langle u_1,v \rangle \text{,}$$
$$\langle Tu_2, v\rangle = -\frac{1}{2}\langle u_2,v\rangle $$
for all $v\in H^1_{\#}(\alpha,Y)$.

Let $Q_1, Q_2$ be the orthogonal projections of $H^1_{\#}(\alpha,Y)$ onto $W_1$ and $W_2$ respectively, and define $P_1 := Q_1 + P_{1/2} \text{,  } P_2:= Q_2 $.  Here $P_{1/2}$ is the projection onto the one dimensional subspace spanned by the function $\Pi\in H_{\#}^1(\alpha,Y)$. Then $\{ P_1, P_2\} \cup \{ P_{\mu_i}\}_{-\frac{1}{2} < \mu_i < \frac{1}{2}}$ is an orthogonal family of projections, and
$$\langle P_1u + P_2u + \sum_{-\frac{1}{2} < \mu_i < \frac{1}{2}} P_{\mu_i}u, v\rangle = \langle u,v\rangle$$
for all $u,v \in H^1_{\#}(\alpha,Y)$.

The representation for $T^\alpha_k$ is given by the following lemma
\begin{lemma}{{\rm Representation of bilinear form, \cite{RobertRobert1}}}\\
\label{T_z spectrum}
The linear operator $T^\alpha_k: H^1_{\#}(\alpha,Y) \longrightarrow H^1_{\#}(\alpha,Y)$ associated with the bilinear form $B_k$ is  is given by
$$\langle T^\alpha_k u,v \rangle = \langle k P_1u + P_2u + \sum_{-\frac{1}{2} < \mu_i < \frac{1}{2}}[k(1/2 + \mu_i) + (1/2-\mu_i)] P_{\mu_i}u, v\rangle$$
for all $u,v\in H^1_{\#}(\alpha,Y)$.
\label{repbilinearlemma}
\end{lemma}

We analyze the Bloch spectra by  by developing a power series in $z=\frac{1}{k}$ about $z=0$ for the spectrum of the family of operators given by 
\begin{equation}
\begin{array}{lcl}
A^\alpha(z)& = &  (zP_1 + P_2 + z\sum_{-\frac{1}{2} < \mu_i < \frac{1}{2}}[(1/2 + \mu_i) + z(1/2-\mu_i)]^{-1} P_{\mu_i})(-\Delta_{\alpha})^{-1}.\\
\end{array}
\label{aalpharepresentation}
\end{equation}
Here the operator $(-\Delta_{\alpha})^{-1}$ defined for all $\alpha\in Y^\ast$ is given by
\begin{eqnarray}
(-\Delta_{\alpha})^{-1} u(x)=-\int_Y G^\alpha(x,y) u(y)\,dy,
\label{inverselaplacian22}
\end{eqnarray}
where $G^\alpha(x,y)$ is the Greens function for the quasi-periodic Laplace operator.

The operator  $A^{\alpha}(z)$ is self-adjoint for $k\in\mathbb{R}$ and is a family of bounded operators taking $L^2_{\#}(\alpha,Y)$ into itself and we have \cite{RobertRobert1}
\begin{lemma}
\label{inverseoperator1}{\em Spectrum of $A^\alpha(z)$}\\
\begin{enumerate}
\item
$A^{\alpha}(z)$ is holomorphic on $\Omega_0 := \mathbb{C} \setminus S$. 
Where $S=\cup_{i\in\mathbb{N}} z_i$ is the collection of points $z_i=(\mu_i+1/2)/(\mu_i-1/2)$ on the negative real axis associated with
the eigenvalues $\{\mu_i\}_{i\in\mathbb{N}}$. The set $S$ consists of poles  of $A^\alpha(z)$ with only one accumulation point at $z=-1$.
\item For $z\in \Omega_0$ the spectrum of $A^\alpha(z)$ denoted by $\sigma(A^\alpha(z))$ consists of eigenvalues $\beta_j^\alpha(z)$ of finite multiplicity with a possible accumulation point $z=0$.
\end{enumerate}
\end{lemma}

We have the following theorem 
\begin{theorem}{\rm \cite{RobertRobert1}}
\label{extension6}
The the Bloch eigenvalue problem \eqref{Eigen1localBloch} for $-\nabla(k(1-\chi_{D})+\chi_D)\nabla$ defined on $H_{\#}^1(\alpha,Y)$ can be extended for values of the coupling constant $k$ with $k^{-1}\in\mathbb{C}\setminus S$, such that for each $\alpha\in Y^\star$  the Bloch eigenvalues are of finite multiplicity and denoted by $\lambda_j(k,\alpha)=(\beta_j^{\alpha}(k^{-1}))^{-1}$, $j\in \mathbb{N}$ and the band structure
\begin{equation}
\lambda_j(k,\alpha)=\omega^2,\hbox{ $j\in\mathbb{N}$}
\label{DispersionRelations}
\end{equation}
extends to complex coupling constants $k$ with $k^{-1}\in\mathbb{C}\setminus S$.
\end{theorem}

Analysis provided in \cite{RobertRobert1} shows that upper bounds  introduced the sections \ref{radiusgeneralshape}   and \ref{radiusmultiplescatterers} also apply to the set of quasi-periodic source free resonances and $S$ is bounded uniformly away from $0$ for all $\alpha\in Y^\ast$, i.e.,
\begin{eqnarray}
 \max_i \{z_i\}\leq\frac{\mu^\ast+1/2}{\mu^\ast-1/2}=z^\ast<0.
\label{bdsonSS}
\end{eqnarray}

Let $\beta^\alpha_0 \in \sigma(A^{\alpha}(0))$ with spectral projection $P(0)$, and let $\Gamma$ be a closed contour in $\mathbb{C}$ enclosing $\beta^\alpha_0$ but no other $\beta^\alpha_0 \in \sigma(A^{\alpha}(0))$.
The spectral projection associated with $\beta^\alpha (z) \in \sigma(A^{\alpha}(z))$ for $\beta^\alpha(z) \in \text{int}(\Gamma)$ is denoted by $P(z)$. As before we suppose all elements $\beta^\alpha(z) \in \text{int}(\Gamma)$ converge to $\beta^\alpha_0$ as $z\rightarrow0$. This collection is known as the eigenvalue group associated with $\beta^\alpha_0$.  We write $M(z) = P(z)L^2_{\#}(\alpha,Y)$ and suppose for the moment that $\Gamma$ lies in the resolvent of $A^\alpha(z)$ and $\text{dim}(M(0)) =\text{dim}(M(z))= m$, noting that  Theorems \ref{separationandraduus-alphanotzero} and \ref{separationandraduus-alphazero}  provide explicit conditions for when this holds true.  Since $A^{\alpha}(z)$ is analytic in a neighborhood of the origin we  write
\begin{equation}
A^{\alpha}(z) = A^{\alpha}(0) + \sum_{n=1}^{\infty} z^nA^{\alpha}_n.
\end{equation}
Define the resolvent of $A^{\alpha}(z)$ by
$$R(\zeta, z) = (A^{\alpha}(z) - \zeta )^{-1}\text{,}$$
and expanding successively in Neumann series and power series we have the identity

\begin{equation}
\begin{array}{lcl}
R(\zeta,z) & = & R(\zeta,0)[I + (A^{\alpha}(z) - A^{\alpha}(0))R(\zeta,0)]^{-1} \\
\\
& = & R(\zeta,0)+\sum_{p=1}^\infty [-(A^{\alpha}(z) - A^{\alpha}(0))R(\zeta,0)]^p\\
\\
& = & R(\zeta,0) + \sum_{n=1}^{\infty} z^nR_n(\zeta)\text{,}
\end{array}
\label{foursix6}
\end{equation}

where
$$R_n(\zeta) = \sum_{k_1 + \ldots k_p = n, k_j \geq 1} (-1)^pR(\zeta,0)A^{\alpha}_{k_1}R(\zeta,0)A^{\alpha}_{k_2}\ldots R(\zeta,0)A^{\alpha}_{k_p}$$
for $n\geq 1$.

Application of  the contour integral formula for spectral projections \cite{SzNagy}, \cite{TKato1}, \cite{TKato2} delivers the expansion for the spectral projection
\begin{equation}
\begin{array}{lcl}
P(z) & = & -\frac{1}{2\pi i} \oint_{\Gamma} R(\zeta, z)d\zeta\\
\\
& = & P(0) + \sum_{n=1}^\infty z^n P_n
\end{array}
\label{Project16}
\end{equation}
where $P_n =  -\frac{1}{2\pi i} \oint_{\Gamma} R_n(\zeta)d\zeta$.  Proceeding as before we get 
\begin{equation}
(A^{\alpha}(z) - \beta^\alpha_0)P(z) = - \frac{1}{2 \pi i}\oint_{\Gamma} (\zeta - \beta^\alpha_0)R(\zeta,z)d\zeta \text{,}
\label{cauchyformula6}
\end{equation}
and the power series representation for the operator $A^\alpha(z)$ follows from \eqref{foursix6}.

The operator $A^\alpha(z)$ is selfadjoint and bounded for $z\in\mathbb{R}\not\in S$ see, \cite{RobertRobert1}.
So for $z\in \mathbb{R}\setminus S$ there is a complete orthonormal system of eigenfunctions $\{\varphi_i(z)\}_{i\in\mathbb{N}}$ in $L^2_0(Y)$ and eigenvalues $\{\beta^\alpha_i(z)\}_{i\in\mathbb{N}}$ such that 
\begin{equation}
\label{complete}
\beta^\alpha_i(z)=(A^\alpha(z)\varphi_i(z),\varphi_i(z)), \hbox{   for   $z\in \mathbb{R}\setminus S$}.
\end{equation}
Now for $\beta^\alpha_0$ such that $\beta^\alpha_0=(A^\alpha(0)\varphi_i(0),\varphi_i(0))$ and $\beta^\alpha_i(z)$ associated with the eigenvalue group corresponding to $\beta^\alpha_0$ inside $\Gamma$ we have $P(z)\varphi_i(z)=\varphi_i(z)$ and from \eqref{cauchyformula6}
\begin{eqnarray}
&&\beta^\alpha_i(z) - \beta^\alpha_0 =((A^\alpha(z) - \beta^\alpha_0)P(z)\varphi_i(z),\varphi_i(z))\nonumber\\
&&=- (\frac{1}{2 \pi i}\oint_{\Gamma} (\zeta - \beta^\alpha_0)R(\zeta,z)d\zeta\varphi_i(z),\varphi_j(z)).
\label{fourten6}
\end{eqnarray}
As before the eigenfunctions $\varphi_i(z)$ are analytic in $z$, see \cite{KatoPerturb}. 

We develop a series for the eigenvalues suitable for the analysis. We apply \eqref{aalpharepresentation} and write
\begin{equation}
\begin{array}{lcl}
R(\zeta,z) & = & R(\zeta,0)+\sum_{p=1}^\infty [-(A^\alpha(z) - A(0))R(\zeta,0)]^p\\
\\
& = & R(\zeta,0) + \sum_{n=1}^{\infty} z^n({\mathcal{N}}^\alpha(\zeta,z))^n\text{,}
\end{array}
\label{foursixalternate6}
\end{equation}
where
\begin{equation}
\label{caligraphN6}
\mathcal{N}^\alpha(\zeta,z) =  [(P_1u + \sum_{-\frac{1}{2} < \mu_i < \frac{1}{2}}[(1/2 + \mu_i) + z(1/2-\mu_i)] ^{-1}P_{\mu_i})(-\Delta_\alpha)^{-1}R(\zeta,0)]
\end{equation}
for $n\geq 1$.
Equation \eqref{fourten6} together with \ref{foursixalternate6} deliver a series representation formula for  Neumann eigenvalues $\nu_i(k)=(\beta_i(\frac{1}{k}))^{-1}$ for real $k$ in a neighborhood of $\infty$. Substituting \eqref{foursixalternate6} into \eqref{fourten6} and manipulation  yields
\begin{equation}
{\beta}^\alpha_i(z) = \beta^\alpha_0 + \sum_{n=1}^\infty z^n\beta^n_i(z),
\label{foureleven6}
\end{equation}
where
\begin{equation}
{\beta}^n_i(z) = - \frac{1}{2\pi i} \left(\oint_{\Gamma}(\zeta-\beta_0)(\mathcal{N}^\alpha(\zeta,z))^nd\zeta\varphi_i(z),\varphi_i(z)\right);\hbox{ $n\geq 1$}.
\label{fourtwelve6}
\end{equation}
for $z$ in an interval containing $z=0$.

\section{Neuman and Bloch spectrum in the High Contrast Limit}
\label{limitspeczero}
We identify the spectrum of the limiting operator $A(0)$.  Using the representation
\begin{equation}
A(z) = (zP_1 + P_2 + z\sum \limits_{-\frac{1}{2} < \mu_i < \frac{1}{2}} [(1/2 + \mu_i) + z(1/2-\mu_i)]P_{\mu_i})(-\Delta_{N})^{-1} \text{,}
\end{equation}
we see that
\begin{equation}
A(0) = P_2(-\Delta_{N})^{-1},
\end{equation}
and denote the spectrum of $A(0)$ by $\sigma(A(0))$. 

 To begin consider the Dirichlet eigenvalues of the Laplace operator on $D$ associated with the spectral problem $-\Delta\psi=\delta\psi$, $\psi\in H^1_0(D)$ and denote the spectrum by $\sigma(-\Delta_D)$. The subset of Dirichlet eigenvalues associated with eigenfunctions having zero mean over $D$ is denoted by $\{\delta_j'\}_{j\in\mathbb{N}}$ and the set of Dirichlet eigenvalues associated with eigenfunctions with nonzero mean is denoted by $\{\delta_j^\ast\}_{j\in\mathbb{N}}$.  Next we introduce the sequence of numbers $\{\nu_j\}_{j\in\mathbb{N}}$ given by the positive roots $\nu$ of the  spectral function $S(\nu)$ defined by
\begin{equation}
S(\nu)=\nu\sum \limits_{i\in \mathbb{N}} \frac{ a^2_i}{\nu-\delta^*_i}-1,
\label{roots}
\end{equation}
where the coefficients  $a_j=|\int_D\psi_j\,dx|$ are integrals of eigenfunctions $\psi_j$ corresponding to the Dirichlet eigenvalues $\delta^*_j$. The explicit characterization of $\sigma(A(0))$ is given by the following theorem.
\begin{theorem}
\label{equiv2}
$\sigma(A(0)) = \{{\delta'_j}^{-1}\}_{j\in\mathbb{N}}\cup\{{\nu_j}^{-1}\}_{j\in\mathbb{N}}$.
\end{theorem}

To  establish the theorem we first  show that the eigenvalue problem
\begin{equation}
P_2(-\Delta_{N})^{-1}u=\lambda u
\label{limitneumanproblem1}
\end{equation}
 with $\lambda \in \sigma(A^{\alpha}(0))$ and eigenfunction $u\in L^2_0(Y)$ is  equivalent to finding $\lambda$ and $u\in W_2$ for which
\begin{equation}
(u, v)= \lambda \langle u, v \rangle, \hbox{ for all $v\in W_2$}.
\label{eqiveigenH01}
\end{equation}
To see the equivalence note that we have $u=P_2u$ and for $v \in \mathcal{H}$,
\begin{equation}
\label{fivefour}
\begin{array}{lcl}
\langle P_2(-\Delta_N)^{-1}u, v \rangle =  \lambda \langle u, v \rangle=\lambda \langle P_2u, v \rangle \end{array}
\end{equation}
hence
\begin{equation}
\label{fivefour2}
\begin{array}{lcl}
\langle (-\Delta_{N})^{-1}u, P_2v\rangle= \lambda \langle u, P_2v \rangle.
\end{array}
\end{equation}
Since $\langle (-\Delta_{N})^{-1}u, v\rangle =(u, v)$ for any $u\in L^2_0(Y)$ and $v \in \mathcal{H}$, equation  \eqref{fivefour2} becomes
\begin{equation}
\label{projectedalpnanotzero}
(u, P_2v)= \lambda \langle u, P_2v \rangle,
\end{equation}
and the equivalence follows noting that $P_2$ is the  projection of $\mathcal{H}$ onto $W_2$.

To conclude we  show that the set of eigenvalues for \eqref{limitneumanproblem1} is given by 
$\{{\delta'_j}^{-1}\}_{j\in\mathbb{N}}\cup\{\nu_j^{-1}\}_{j\in\mathbb{N}}$.
We see that $u\in W_2$ and from \eqref{definitionW2periodic} we have the dichotomy: $\int_D \tilde{u}dx=0$ and $u=\tilde{u}\in \tilde{H}^1_0(D)$ or  $\int_D \tilde{u}dx\not=0$ and $u=\tilde{u}-\gamma 1_Y$ with $\gamma=\int_D\tilde{u} dx$.  It is evident for the first case that the eigenfunction $u\in\tilde{H}^1_0(D)$ and for $v\in W_2$ given by 
\begin{equation}
\label{definitionW2periodicdefine}
v = \tilde{v} - \left( \int_D \tilde{v} dx \right) 1_Y \; \hbox{for} \;  \tilde{v} \in \tilde{H}^1_0(D)
\end{equation}
that problem \eqref{eqiveigenH01} becomes
\begin{equation}
\int_Du\overline{\tilde{v}} = \lambda \int_D\nabla u\cdot\nabla\overline{\tilde{v}}, \hbox{ for all $\tilde{v}\in \tilde{H}^1_0(D)$},
\label{eqiveigenW2first}
\end{equation}
and we conclude that $\tilde{u}$ is a Dirichlet eigenfunction with zero average over $D$ so $\lambda\in\{{\delta'_j}^{-1}\}_{j\in\mathbb{N}}$. While for the second, we have $u\in W_2$ and again
\begin{equation}
\int_Du\overline{\tilde{v}} = \lambda \int_D\nabla u\cdot\nabla\overline{\tilde{v}}, \hbox{ for all $\tilde{v}\in \tilde{H}^1_0(D)$}.
\label{eqiveigenW2nd}
\end{equation}
Writing $u=\tilde{u}-\gamma 1_Y$, $\lambda=\nu^{-1}$ and integration by parts in \eqref{eqiveigenW2nd} shows that $\tilde{u}\in\tilde{H}^1_0(D)$ is the solution of
\begin{eqnarray}
\Delta \tilde{u}+\nu \tilde{u}=\nu\gamma \hbox{ for $x\in D$}.
\label{eigenw2}
\end{eqnarray}
We normalize  $\tilde{u}$ so that $\gamma = \int_D\tilde{u}dx=1$  and write
\begin{equation}
\label{series}
\tilde{u} = \sum \limits_{j=1}^{\infty} c_j \psi_j
\end{equation}
where, $\psi_j$ are the Dirichlet eigenfunctions of $-\Delta_D$ associated with eigenvalue $\delta_j$ extended by zero to $Y$.  Substitution of \eqref{series} into \eqref{eigenw2} gives
\begin{equation}
\label{fourierequation}
\sum \limits_{j=1}^{\infty} (-\delta_j + \nu) c_j \psi_j = \nu.
\end{equation}
Multiplying both sides of \eqref{fourierequation} by $\overline{\psi_k}$ over $Y$ and orthonormality of $\{\psi_j\}_{j \in \mathbb{N}}$, shows that $\tilde{u}$ is given by
\begin{equation}
\label{finalformula}
\tilde{u} = \nu \sum \limits_{k\in \mathbb{N}} \frac{ \int_D \overline{{\psi}_k}}{\nu-\delta^*_k} \psi_k\text{,}
\end{equation}
where $\delta_k^*$ correspond to Dirichlet eigenvalues associated with eigenfunctions for which $\int_D\,\psi_k\,dx\not=0$. 
To calculate $\nu$, we integrate both sides of \eqref{finalformula} over $D$ to recover the identity
\begin{equation}
\nu \sum \limits_{k\in \mathbb{N}} \frac{ a^2_k}{\nu-\delta^*_k}-1=0.
\label{theIdentity}
\end{equation}
It follows from  \eqref{theIdentity} that $\lambda\in\{\nu_i^{-1}\}_{i\in\mathbb{N}}$ and the proof of Theorem \ref{equiv2} is complete.

We now recover the spectrum of the limiting operator $A^{\alpha}(0)$ when $\alpha \in Y^\ast$.  Using the representation
\begin{equation}
\label{alphaformula}
A^{\alpha}(z) = (zP_1 + P_2 + z\sum \limits_{-\frac{1}{2} < \mu_i < \frac{1}{2}} [(1/2 + \mu_i) + z(1/2-\mu_i)]P_{\mu_i})(-\Delta_{\alpha})^{-1} \text{,}
\end{equation}
we see that
\begin{equation}
A^{\alpha}(0) = P_2(-\Delta_{\alpha})^{-1}.
\end{equation}
The following theorem provides the explicit characterization of the spectrum  $\sigma(A^\alpha(0))$ for $\alpha\not=0$. 
\begin{theorem}
\label{equiv12}{\em Limit spectrum for quasi-periodic problem \cite{RobertRobert1}.}\\
 $\sigma(A^{\alpha}(0)) = \{{\delta'_j}^{-1}\}_{j\in\mathbb{N}}\cup\{(\delta_j^\ast)^{-1}\}_{j\in\mathbb{N}}=\sigma(-\Delta_D^{-1})$.
\end{theorem}
To conclude we recover the limit spectrum $\sigma(A^0(0))$ for the periodic problem, see \cite{RobertRobert1}.
\begin{theorem}
\label{equiv22}
$\sigma(A^{0}(0)) = \{{\delta'_j}^{-1}\}_{j\in\mathbb{N}}\cup\{{\nu_j}^{-1}\}_{j\in\mathbb{N}}$.
\end{theorem}
Theorems \ref{equiv2}, \ref{equiv12}, and \ref {equiv22} are in accord with   Lemma 2.3 parts (b) and (c) of \cite{HempelLienau}.

\section{Radius of Convergence and Separation of Spectra}
\label{radius}


Fix an inclusion geometry specified by the domain $D$. Suppose first $\alpha\in Y^\star$ and $\alpha\not =0$. Recall from Theorem \ref{equiv12}  that the spectrum of $A^\alpha(0)$ is $\sigma(-\Delta_D^{-1})$. Take $\Gamma$ to be a closed contour in $\mathbb{C}$ containing an eigenvalue  $\beta^\alpha_0$ in $\sigma(-\Delta^{-1}_{D})$ but no other element of $\sigma(-\Delta^{-1}_{D})$, see Figure \ref{spectrum}. Define $\hat{d}$ to be the distance between $\Gamma$ and $ \sigma(-\Delta^{-1}_{D})$, i.e., 
\begin{eqnarray}
\hat{d}={\rm{dist}}(\Gamma,\sigma(-\Delta^{-1}_{D})=\inf_{\zeta\in\Gamma}\{{\rm{dist}}(\zeta,\sigma(-\Delta^{-1}_{D})\}.
\label{dist}
\end{eqnarray}
The only component of the spectrum of $A^\alpha(0)$ inside $\Gamma$  is $\beta^\alpha_0$ and we denote this   by $\Sigma'(0)$. The part of the spectrum of $A^\alpha(0)$ in the  domain exterior to $\Gamma$ is denoted by $\Sigma''(0)$ and $\Sigma''(0)=\sigma(-\Delta^{-1}_{D})\setminus \beta^\alpha_0$. The invariant subspace of $A^\alpha(0)$ associated with $\Sigma'(0)$ is denoted by $M'(0)$ with $M'(0)=P(0)L^2_{\#}(\alpha,Y)$ .

Let $-1/2<\mu^\ast$ denote the lower bound on the quasi-periodic resonance eigenvalues for the domain $D$. It is noted that in the sequel a wide class of domains are identified for which there exist lower bound s on both quasi-periodic resonances and electro static source free resonances.  The corresponding upper bound on the set $z\in S$ for which $A^\alpha(z)$ is not invertible  is given by 
\begin{eqnarray}
z^\ast=\frac{\mu^\ast+1/2}{\mu^\ast-1/2}<0,
\label{upperonS}
\end{eqnarray}
see \eqref{bdsonS}.
Now set
\begin{equation}
r^*=\frac{|\alpha|^2\hat{d}|z^\ast|}{\frac{1}{1/2-\mu^\ast}+|\alpha|^2\hat{d}}.
\label{radiusalphanotzero}
\end{equation}
\begin{theorem}{\rm Separation of spectra and radius of convergence for $\alpha\in Y^\star$, $\alpha\not=0$.}\\
\label{separationandraduus-alphanotzero}
The following properties  hold for inclusions with domains $D$ that satisfy \eqref{upperonS}:
\begin{enumerate}
\item If $|z|<r^*$ then $\Gamma$ lies in the resolvent of both $A^\alpha(0)$ and $A^\alpha(z)$ and thus separates the spectrum of $A^\alpha(z)$ into two parts given by the component of spectrum of $A^\alpha(z)$ inside $\Gamma$ denoted by $\Sigma'(z)$ and the component exterior to $\Gamma$  denoted by $\Sigma''(z)$. $\Sigma'(z)$ consists of the eigenvalue group $\beta^\alpha(z)$ associated with $\beta_0^\alpha$. The invariant subspace of $A^\alpha(z)$ associated with $\Sigma'(z)$ is denoted by $M'(z)$ with $M'(z)=P(z)L^2_{\#}(\alpha,Y)$.

\item The projection $P(z)$ is holomorphic for $|z|<r^*$ and $P(z)$ is given by
\begin{eqnarray}
P(z)=\frac{-1}{2\pi i}\oint_\Gamma R(\zeta,z)\,d\zeta.
\label{formula}
\end{eqnarray}
\item The spaces $M'(z)$ and $M'(0)$ are isomorphic for $|z|<r^*$.
\item The  series \eqref{foureleven6} converges uniformly for $z\in\mathbb{R}$ with $|z|<r^*$.

\end{enumerate}
\end{theorem}

\begin{figure} 
\centering
\begin{tikzpicture}[xscale=0.70,yscale=0.70]
\draw [-,thick] (-6,0) -- (6,0);
\draw [<->,thick] (0,0) -- (1.5,0);
\node [above] at (1,0) {$\hat{d}$};
\draw [<->,thick] (1.5,0) -- (3,0);
\node [above] at (2,0) {$\hat{d}$};
\draw (-4,0.2) -- (-4.0, -0.2);
\node [below] at (-4,0) {$\hat{\beta}^\alpha_0$};
\draw (0,0.2) -- (0, -0.2);
\node [below] at (0,0) {$\beta^\alpha_0$};
\draw (3,0.2) -- (3, -0.2);
\node [below] at (3,0) {$\check{\beta}^\alpha_0$};
\draw (0,0) circle [radius=1.5];
\node [right] at (1.1,1.1) {$\Gamma$};
\end{tikzpicture} 
\caption{$\Gamma$}
 \label{spectrum}
\end{figure}

Suppose now $\alpha=0$. Recall from Theorem \ref{equiv22} that the limit spectrum for $A^0(0)$ is
$\sigma(A^0(0)) =  \{{\delta'_j}^{-1}\}_{j\in\mathbb{N}}\cup\{{\nu_j}^{-1}\}_{j\in\mathbb{N}}$. For this case take $\Gamma$ to be the closed contour in $\mathbb{C}$ containing an eigenvalue $\beta_0^0$ in $\sigma(A^0(0))$ but no other element of
$\sigma(A^0(0))$ and define 
\begin{eqnarray}
\label{dforperiodic}
\hat{d}=\inf_{\zeta\in\Gamma}\{\rm{dist}(\zeta,\sigma(A^0(0)))\}.
\end{eqnarray}
Suppose the lowest quasi-periodic resonance eigenvalue for the domain $D$ lies inside $-1/2<\mu^\ast<0$ and the corresponding upper bound on $S$ is given by 
\begin{eqnarray}
z^\ast=\frac{\mu^\ast+1/2}{\mu^\ast-1/2}<0.
\label{upperonSzero}
\end{eqnarray}
Set
\begin{equation}
r^*=\frac{4\pi^2\hat{d}|z^\ast|}{\frac{1}{1/2-\mu^\ast}+4\pi^2\hat{d}}.
\label{radiusalphazero}
\end{equation}
\begin{theorem}{\rm Separation of spectra and radius of convergence for $\alpha=0$.}
\label{separationandraduus-alphazero}
\\
The following properties  hold for inclusions with domains $D$ that satisfy \eqref{upperonSzero}: 
\begin{enumerate}
\item If $|z|<r^*$ then $\Gamma$ lies in the resolvent of both $A^0(0)$ and $A^0(z)$ and thus separates the spectrum of $A^0(z)$ into two parts given by the component of spectrum of $A^0(z)$ inside $\Gamma$ denoted by $\Sigma'(z)$ and the component exterior to $\Gamma$  denoted by $\Sigma''(z)$. $\Sigma'(z)$ consists of the eigenvalue group $\beta^0(z)$ associated with $\beta_0^0$. The invariant subspace of $A^0(z)$ associated with $\Sigma'(z)$ is denoted by $M'(z)$ with $M'(z)=P(z)L^2_{\#}(0,Y)$.

\item The projection $P(z)$ is holomorphic for $|z|<r^*$ and $P(z)$ is given by
\begin{eqnarray}
P(z)=\frac{-1}{2\pi i}\oint_\Gamma R(\zeta,z)\,d\zeta.
\label{formula}
\end{eqnarray}
\item The spaces $M'(z)$ and $M'(0)$ are isomorphic for $|z|<r^*$.
\item The series \eqref{foureleven6} converges uniformly for $z\in\mathbb{R}$ with $|z|<r^*$.
\end{enumerate}
\end{theorem}

Now consider the Neumann spectrum. Recall from Theorem \ref{equiv2} that the limit spectrum for $A(0)$ is $\sigma(A(0)) =  \{{\delta'_j}^{-1}\}_{j\in\mathbb{N}}\cup\{{\nu_j}^{-1}\}_{j\in\mathbb{N}}$. For this case take $\Gamma$ to be the closed contour in $\mathbb{C}$ containing an eigenvalue $\beta_0$ in $\sigma(A(0))$ but no other element of
$\sigma(A(0))$ and define 
\begin{eqnarray}
\label{dforNeumann}
\hat{d}=\inf_{\zeta\in\Gamma}\{\rm{dist}(\zeta,\sigma(A(0)))\}.
\end{eqnarray}
Suppose the lowest quasi-periodic resonance eigenvalue for the domain $D$ lies inside $-1/2<\mu^\ast<0$ and the corresponding upper bound on $S$ is given by 
\begin{eqnarray}
z^\ast=\frac{\mu^\ast+1/2}{\mu^\ast-1/2}<0.
\label{upperonSNeumann}
\end{eqnarray}
Set
\begin{equation}
r^*=\frac{\pi^2\hat{d}|z^\ast|}{\frac{1}{1/2-\mu^\ast}+\pi^2\hat{d}}.
\label{radiusalphazero}
\end{equation}
\begin{theorem}{\rm Separation of Neumann spectra and radius of convergence.}
\label{separationandraduus-Neumann}
\\
The following properties  hold for inclusions with domains $D$ that satisfy 
\eqref{upperonSNeumann}: 
\begin{enumerate}
\item If $|z|<r^*$ then $\Gamma$ lies in the resolvent of both $A(0)$ and $A(z)$ and thus separates the spectrum of $A(z)$ into two parts given by the component of spectrum of $A(z)$ inside $\Gamma$ denoted by $\Sigma'(z)$ and the component exterior to $\Gamma$  denoted by $\Sigma''(z)$. $\Sigma'(z)$ consists of the eigenvalue group $\beta(z)$ associated with $\beta_0$. The invariant subspace of $A(z)$ associated with $\Sigma'(z)$ is denoted by $M'(z)$ with $M'(z)=P(z)L^2_{0}(Y)$.

\item The projection $P(z)$ is holomorphic for $|z|<r^*$ and $P(z)$ is given by
\begin{eqnarray}
P(z)=\frac{-1}{2\pi i}\oint_\Gamma R(\zeta,z)\,d\zeta.
\label{formula}
\end{eqnarray}
\item The spaces $M'(z)$ and $M'(0)$ are isomorphic for $|z|<r^*$.
\item The series \eqref{foureleven} converges uniformly for $z\in\mathbb{R}$ with $|z|<r^*$.
\end{enumerate}
\end{theorem}

We have the estimate for the error incurred when only finitely many terms of the series \ref{foureleven} and \ref{foureleven6} are calculated.
\begin{theorem}{\rm Error estimates for the eigenvalue expansion}.\\
\label{errorestimatesforexpansions}
\begin{enumerate}
\label{errorestm}
\item Let $\alpha \neq 0$, and suppose $D$, $z^*$, and $r^*$ are as in Theorem \ref{separationandraduus-alphanotzero}. Then the following error estimate for the series \eqref{foureleven6} holds for $z\in\mathbb{R}$ and $|z|<r^*$:
\begin{equation}
\label{errorestalpha}
\left |{\beta}^{\alpha}_i(z) - \sum \limits_{n = 0}^{p} z^n \beta^n_i(z) \right | \leq \frac{\hat{d}|z|^{p+1}}{(r^*)^p(r^* - |z|)}.
\end{equation}
\item Let $\alpha = 0$, and suppose $D$, $z^*$, and $r^*$ are as in Theorem \ref{separationandraduus-alphazero}. Then the following error estimate for the series \eqref{foureleven6} holds for $z\in\mathbb{R}$ and $|z|<r^*$:
\begin{equation}
\left |{\beta}^{0}_i(z) - \sum \limits_{n = 0}^{p} z^n \beta^n_i(z) \right | \leq \frac{\hat{d}|z|^{p+1}}{(r^*)^p(r^* - |z|)}.
\label{errorestalphaequalzero}
\end{equation}
\item Consider the Neumann spectrum and suppose $D$, $z^*$, and $r^*$ are as in Theorem \ref{separationandraduus-Neumann}. Then the following error estimate for the series \eqref{foureleven} holds for $z\in\mathbb{R}$ and $|z|<r^*$:
\begin{equation}
\left |{\beta}_i(z) - \sum \limits_{n = 0}^{p} z^n \beta^n_i(z) \right | \leq \frac{\hat{d}|z|^{p+1}}{(r^*)^p(r^* - |z|)}.
\label{errorestNeumann}
\end{equation}
\end{enumerate}
\end{theorem}

\begin{remark}
\label{Uniformini}
Note that the estimates \eqref{errorestalpha} through \eqref{errorestNeumann} are uniform in the index $i$ for elements in the eigenvalue group restricted to the real axis for $|z|<r^\ast$.
\end{remark}

The Theorems \ref{separationandraduus-alphanotzero} and \ref{separationandraduus-alphazero} parts 1 through 3 and explicit convergence radii for power series representation for the eigenvalue group are proved in \cite{RobertRobert1}. The proofs of part 4 of Theorems \ref{separationandraduus-alphanotzero} and  \ref{separationandraduus-alphazero} together with Theorems \ref{separationandraduus-Neumann}
and   \ref{errorestm} are given in section \ref{derivation}.

\section{Radius of Convergence and Separation of Spectra for Periodic Scatterers of General Shape}
\label{radiusgeneralshape}

In this section we describe general  conditions on photonic or phononic crystals that guarantee  a power series representation for spectral bands and for which all Theorems in section \ref{radius} hold. Consider an inclusion domain $D=\cup_{i=1}^N D_i$. Suppose we can surround each $D_i$ by a buffer layer $R_i$ so that each inclusion $D_i$ together with its buffer does not intersect with the any of the other buffered inclusions, i.e., $D_i\cup R_i\cap D_j\cup R_j=\emptyset$, $i\not=j$. The set of such inclusion domains will be called {\em buffered dispersions of inclusions}, see Figure \ref{plane2}.  We denote the operator norm for the Dirichlet to Neumann map for each inclusion by $\Vert DN_i\Vert$ and the Poincare constant for each buffer layer by $C_{R_i}$ and we state the following theorem.
\begin{theorem}{\rm Convergent power series for the Bloch and Neumann spectra for buffered dispersions of inclusions.}
All 2 dimensional photonic and $d$ dimensional acoustic crystals $(d=2,3)$ made from buffered dispersions with constants $C_{R_i}$ and $\Vert DN_i\Vert$ that satisfy 
\begin{equation}
\max_i\{(1+C_{R_i})\Vert DN_i\Vert\}<\infty,
\label{finitednpoincare}
\end{equation}
have Bloch and Neumann spectra described by convergent power series for real values of the contrast within a neighborhood of $z=1/k=0$.  The radii of convergence is controlled by the values $\Vert DN_i\Vert$  and  $C_{R_i}$, $i=1,\ldots,N$ and Theorems \ref{separationandraduus-alphanotzero} through \ref{errorestm} apply to these crystals.
\label{bufferconvergence}
\end{theorem}

The theorem follows from an explicit condition on the inclusion geometry that guarantees a lower bound $\mu^\ast$ for  both  Neumann electrostatic resonance spectra and quasi-periodic resonance spectra. The lower bound  depends only upon geometry and is uniform in $\alpha$ for the quasi-periodic spectra. This lower bound provides a positive distance between the origion $z=0$ and the poles of $A^\alpha(z)$ and $A(z)$. The explicit condition is given by the following criterion.
\begin{theorem}
\label{spectralboundonmu}
Let $D \Subset Y$ be a union of simply connected sets (inclusions) $D_i$, $i=1,\ldots,L$  with $C^{1,\gamma}$ boundary.  Consider the  spectrum $\{\mu_i\}_{i\in \mathbb{N}}$ of $T$ restricted to $W_3$  for  $W_3\subset H^1_{\#}(\alpha, Y)$ or $W_3\subset \mathcal{H}$. For either case if there is a $\theta >0$ such that for all $u \in W_3$ 
\begin{equation}
\label{thetaineq}
\| \nabla u\|_{L^2(Y\setminus D)}^2 \geq \theta \| \nabla u \|_{L^2(D)}^2,
\end{equation}
then for $\rho = \min \{ \frac{1}{2}, \frac{\theta}{2} \}$  one has the lower bound
\begin{equation}
\label{thelowerbound}
\min_{i\in\mathbb{N}}\{ \mu_i\}\geq \mu^\ast=\rho-\frac{1}{2}>-\frac{1}{2}.
\end{equation}
\end{theorem}
This theorem is proved for quasi-periodic resonances in  $W_3=H^1_{\#}(\alpha, Y)$, in \cite{RobertRobert1} and its proof follows identical lines for the Neumann electrostatic resonances in $W_3=\mathcal{H}$.

The parameter $\theta$ is a geometric descriptor for $D$  and we define a wide class of crystal geometries to which Theorems \ref{separationandraduus-alphanotzero} through \ref{errorestm} apply. 

\begin{definition}
The class of crystal geometries characterized by inclusions $D$ such that \eqref{spectralboundonmu} holds for a fixed positive value of $\theta$ is denoted by $P_\theta$. 
\label{Ptheta}
\end{definition}

And we have the corollary:

\begin{corollary}
\label{theta}
Theorems \ref{separationandraduus-alphanotzero} through \ref{errorestm}  hold for every inclusion domain $D$ belonging to $P_\theta$.
\end{corollary}

The convergent series representation for buffered dispersions of inclusions now follows from the theorem.
\begin{theorem}
Suppose there is a $\theta>0$ for which
\begin{eqnarray}
\theta^{-1}\geq\max_i\{(1+C_{R_i})\Vert DN_i\Vert\}.
\label{setcriteria}
\end{eqnarray}
then the buffered geometry lies in $P_\theta$.
\label{pthetabuffered}
\end{theorem}
This theorem is established in \cite{RobertRobert1} for  quasi-periodic spectra and an identical proof can be used to prove it for the Neumann electrostatic spectra discussed here.

\section{Radius of Convergence and Separation of Spectra  for Disks}
\label{radiusmultiplescatterers}
We now consider both Neumann and Bloch spectra for crystals discussed in the introduction with each  period cell containing an identical  distribution of $N$ disks $D_i$, $i=1,\ldots, N$ of radius $a$. We suppose that the  smallest distance separating the disks is $t_d>0$. The buffer layers $R_i$ are annuli with inner radii $a$ and outer radii $b=a+t$ where $t\leq t_d/2$ and  is chosen so that the collection of buffered disks lie within the period cell. For this case a suitable constant $\theta$ is computed in \cite{Bruno} and is given by

\begin{equation}
\label{thetaofb}
\theta = \frac{b^2 - a^2}{b^2 + a^2}.
\end{equation}
Since $a<b$, we have that
\begin{equation}
\label{thetabounds}
0 < \theta <1.
\end{equation}
We also note that when $D_i$ are discs of radius $a>0$, we can recover an explicit formula for $d$ from equations \ref{dist}, \ref{dforperiodic}, and \ref{dforNeumann}.  In particular, any eigenvalue $\beta_j^{\alpha}(0)$ of $-\Delta_D^{-1}$, for $\alpha\not=0$. may be written
\begin{equation}
\label{dirichletvaluedisc}
\beta_j^{\alpha}(0) = \left (\frac{\eta_{n,k}}{a} \right )^{-2} \text{,}
\end{equation}
where $\eta_{n,k}$ is the $k$th zero of the $n$th Bessel function $J_n(r)$.  Let $\tilde{\eta}$ be the minimizer of
\begin{equation}
\min \limits_{m,j \in \mathbb{N}} |(\eta_{n,k})^{-2} - (\eta_{m,j})^{-2}|.
\end{equation} 
Then we may choose $\Gamma$ from section \ref{radius} so that
\begin{equation}
\hat{d} = \frac{1}{2}|(\frac{a}{\eta_{n,k}})^2 - (\frac{a}{\tilde{\eta}})^2|.
\label{dDirichlet}
\end{equation}
We apply the explicit form for $\theta$ to obtain a formula for $r^*$ in terms of $a$,$b$, $d$ given above, and $\alpha$.  Recall that $\rho$ from Theorem \ref{spectralboundonmu} is given by $\rho = \min \{\frac{1}{2}, \frac{\theta}{2} \}$.  In light of inequality \eqref{thetabounds}, we have that
\begin{equation}
\label{rhoforab}
\rho = \frac{1}{2} \left ( \frac{b^2-a^2}{b^2+a^2} \right ),
\end{equation}
and we calculate the lower bound $\mu^\ast$:
\begin{equation}
\label{lowermuforab}
\mu^\ast = \rho - \frac{1}{2} = -\frac{a^2}{b^2+a^2}.
\end{equation}
Recalling that
$$z^*= \frac{\mu^\ast + 1/2}{1/2-\mu^\ast} \text{ , }$$
we obtain an explicit radius of convergence $r^*$ in terms of $a$, $b$, $\eta_{n,k}$, $\tilde{\eta}$, and $\alpha$ for $\alpha \neq 0$,
\begin{equation}
\label{ralphaforab}
r^* = \frac{|\alpha|^2|(\frac{a}{\eta_{n,k}})^2 - (\frac{a}{\tilde{\eta}})^2|(b^2-a^2)}{4(b^2+a^2) + |\alpha|^2|(\frac{a}{\eta_{n,k}})^2 - (\frac{a}{\tilde{\eta}})^2|(b^2+3a^2)}.
\end{equation}

When $\alpha=0$  Theorem \ref{equiv22} shows that the limit spectrum consists of a component given by the roots $\nu_{0k}$ of
\begin{eqnarray}
1=N\nu\sum_{k\in\mathbb{N}}\frac{a_{0k}^2}{\nu-(\eta_{0k}/a)^2},
\label{extraspec}
\end{eqnarray}
where $a_{0k}=\int_D u_{0k}\,dx$ are averages of the rotationally symmetric normalized eigenfunctions $u_{0k}$ given by
\begin{eqnarray}
u_{0k}=J_0(r\eta_{0k}/a)/(a\sqrt{\pi}J_1(\eta_{0k})).
\label{rotintavg}
\end{eqnarray}
The other component is comprised of the eigenvalues exclusively associated with mean zero eigenfunctions. The collection of these eigenvalues is given by $\{\cup_{n\not=0,k}(\eta_{nk}/a)^2\}$
The elements $\lambda_{nk}$ of the spectrum $\sigma(A^0(0))$ are given by the set $\{\cup_{n\not=0,k}(\eta_{nk}/a)^2\}\cup\{\cup_k\nu_{0k}\}$. Now fix an element $\lambda_{nk}$ and  let $\tilde{\eta}$ be the minimizer of
\begin{equation}
\min \limits_{m,j \in \mathbb{N}} |(\lambda_{n,k})^{-1} - (\lambda_{m,j})^{-1}|.
\end{equation} 
Then as before we may choose $\Gamma$ from section \ref{radius} so that
\begin{equation}
\hat{d} = \frac{1}{2}|({\lambda_{n,k}}^{-1} - {\tilde{\eta}}^{-1}|
\end{equation}
and in terms of $a$, $b$, $\lambda_{n,k}$, and $\tilde{\eta}$ for $\alpha = 0$:
\begin{equation}
\label{ralphaforabzero}
r^* = \frac{\pi^2|(\lambda_{n,k})^{-1} - {\tilde{\eta}}^{-1}|(b^2-a^2)}{(b^2+a^2) + \pi^2|({\lambda_{n,k}})^{-1} - {\tilde{\eta}}^{-1}|(b^2+3a^2)}.
\end{equation}

Theorem \ref{equiv2} shows that the limit spectrum for the Neumann eigenvalue problem also consists of a component given by the roots $\nu_{0k}$ of \eqref{extraspec}. The elements $\lambda_{nk}$ of the spectrum $\sigma(A(0))$ are the same as for the limit periodic case $\sigma(A^0(0))$ and given by the set $\{\cup_{n\not=0,k}(\eta_{nk}/a)^2\}\cup\{\cup_k\nu_{0k}\}$. Proceeding as before we may choose $\Gamma$ from section \ref{radius} so that
\begin{equation}
\hat{d} = \frac{1}{2}|({\lambda_{n,k}}^{-1} - {\tilde{\eta}}^{-1}|
\label{dNeumann}
\end{equation}
and $r^\ast$  is  given by 
\begin{equation}
\label{ralphaforneumann}
r^* = \frac{\pi^2|(\lambda_{n,k})^{-1} - {\tilde{\eta}}^{-1}|(b^2-a^2)}{4(b^2+a^2) + \pi^2|({\lambda_{n,k}})^{-1} - {\tilde{\eta}}^{-1}|(b^2+3a^2)}.
\end{equation}

The collection of suspensions of $N$ buffered disks is an example of a class of buffered inclusion geometries and collecting results we have the following:

\begin{corollary}
\label{theta2}
For every suspension of buffered disks with $\theta$ given by  \eqref{thetaofb}: Theorem \ref{separationandraduus-alphanotzero} holds with $r^*$ given by \eqref{ralphaforab} for $\alpha\in Y^\star$, $\alpha\not=0$, Theorem \ref{separationandraduus-alphazero} holds with $r^*$ given by \eqref{ralphaforabzero} for $\alpha=0$, and Theorem \ref{separationandraduus-Neumann} holds with $r^\ast$ given by \eqref{ralphaforabzero}. Moreover Theorem \ref{errorestimatesforexpansions} part one holds with $r^\ast$ given by \eqref{ralphaforab} and parts two and three hold for $r^\ast$ given by \eqref{ralphaforabzero} and \eqref{ralphaforneumann} respectively
\end{corollary}

\section{Opening band gaps and persistence of pass bands for $k>1$} 
\label{bandgappassband}

We apply the characterization of $\sigma(A(0))$ given by Theorem \ref{equiv2} together with \eqref{relationbeteigenvals} to recover the high contrast limit of the Neumann spectrum given by
\begin{equation}
 \sigma_N=\{{\delta'_j}\}_{j\in\mathbb{N}}\cup\{{\nu_j}\}_{j\in\mathbb{N}},
\label{Neumanlimitspectra}
\end{equation}
where $\delta'_j$ is the part of the Dirichlet spectra for  $D$ associated with mean zero eigenfunctions and $\nu_i$ are the roots of the spectral function \eqref{roots}. This is precisely the high contrast Numann spectrum described in \cite{HempelLienau}. In what follows we  do not distinguish between the two component parts of the spectrum and write elements of $\sigma_N$ as ${\nu}_j$, $j\in\mathbb{N}$. The Dirichlet spectrum is given by $\sigma(-\Delta_D)=\{{\delta'_j}\}_{j\in\mathbb{N}}\cup\{{\delta^\ast_j}\}_{j\in\mathbb{N}}$ where $\delta_j^\ast$ are Dirichlet eigenvalues associated with eigenfunctions with nonzero mean.  The relation between $\sigma_N$ and $\sigma(-\Delta_D)$ is given by the following theorem. 

\begin{theorem}{\rm Strict interlacing of spectra \cite{HempelLienau}}\\
\label{interlacing}
Given $\delta_j^\ast\in\sigma(-\Delta_D)$ and if $\delta_j^\ast$ is simple then there exist adjacent elements $\nu_j<\nu_{j+1}$ ordered by min-max belonging to $\sigma_N$ such that
\begin{equation}
\nu_j<\delta_j^\ast<\nu_{j+1}.
\label{interlacingspectum}
\end{equation}
\end{theorem}
Theorem \ref{interlacing} insures the existence of a band gap for sufficiently large contrast $k$.
We give an explicit condition on the contrast $k$ that is sufficient to open a band gap in the vicinity of $\delta_{j}^\ast$ together with explicit formulas describing its location and bandwidth. 

\begin{theorem}{Opening a band gap}\\
Consider any crystal geometry belonging to the class $P_\theta$. Suppose $\delta_{j}^\ast$ is simple  then $\delta_j^\ast<\nu_{j+1}$. Set 
$$d_j=\frac{1}{2}dist\left(\{\nu_{j+1}^{-1}\},\sigma_N\setminus\{\nu_{j+1}^{-1}\}\right)$$   and 
\begin{equation}
\label{upperrgeneral-j}
\overline{r}_j =\frac{\pi^2d_j|z^\ast|}{\frac{1}{1/2-\mu^\ast}+\pi^2d_j}.
\end{equation}
Then one has the band gap
\begin{equation}
\sigma(L_k)\cap\left(\delta_{j}^\ast,\nu_{j+1}(1-\frac{\nu_{j+1}d_j}{k\overline{r}_j-1})\right)=\emptyset
\label{explicitgapgeneral}
\end{equation}
if
\begin{equation}
k>\overline{k}_j=\overline{r}_j^{-1}\left(1+\frac{d_j\nu_{j+1}}{1-\frac{\delta_j^\ast}{\nu_{j+1}}}\right).
\label{explicitbandgapgeneral}
\end{equation}
\label{stopbandgeneral}
\end{theorem}
Next we provide an explicit condition on $k$  sufficient for the persistence of a spectral band together with explicit formulas describing its location and bandwidth. 

\begin{theorem}{Persistence of passbands}\\
Consider any crystal geometry belonging to the class $P_\theta$.
Suppose $\delta_{j}^\ast$ is simple then $\nu_j<\delta_j^\ast$. Set 
$$d_j=\frac{1}{2}dist\left(\{(\delta^\ast_{j})^{-1}\},\sigma(-\Delta_D)\setminus\{(\delta^\ast_{j})^{-1}\}\right)$$
and
\begin{equation}
\label{lowerrgeneral-j}
\underline{r}_j =\frac{d\pi^2d_j|z^\ast|}{\frac{1}{1/2-\mu^\ast}+d\pi^2d_j}, \hbox{   for $d=2,3$}.
\end{equation}
Then one has a passband in the vicinity of $\delta_{j}^\ast$ and
\begin{equation}
\sigma(L_k)\supset\left[\nu_j,\delta_{j}^\ast(1-\frac{\delta_{j}^\ast d_j}{k\underline{r}_j-1})\right]
\label{explicitpassbandinclusiongeneral}
\end{equation}
if
\begin{equation}
k>\underline{k}_j=\underline{r}_j^{-1}\left(1+\frac{d_j\delta_{j}^\ast}{1-\frac{\nu_j}{\delta_{j}^\ast}}\right).
\label{explicitpassbandgeneral}
\end{equation}
\label{passbandgeneral}
\end{theorem}

Theorem \ref{stopband} follows from Theorem \ref{stopbandgeneral} on applying \eqref{lowermuforab}, \eqref{ralphaforabzero} and \ref{dNeumann}. Theorem \ref{passband} follows from Theorem \ref{passbandgeneral} on applying \eqref{dDirichlet}, \eqref{lowermuforab} and \eqref{ralphaforab} with $|\alpha|^2=d\pi^2$, $d=2,3$.  (See the proof of Theorem \ref{passbandgeneral}).

We now establish Theorem \ref{stopbandgeneral}.
\begin{proof}
Consider $\nu_{j+1}$  of multiplicity $m<\infty$. Set $\hat{d}=d_j$ with$$d_j=\frac{1}{2}dist\left(\{\nu_{j+1}^{-1}\},\sigma_N\setminus\{\nu_{j+1}^{-1}\}\right)$$ 
and $r^\ast=\overline{r}_j$ and apply Theorem \ref{separationandraduus-Neumann} so that any element $\beta_i(z)$, $1\leq i\leq \ell\leq m$ in the eigenvalue group has  series representation given by
\begin{equation}
{\beta}_i(z) = \nu_{j+1}^{-1} + \sum_{n=1}^\infty z^n\beta^n_i(z),
\label{seriesforNeumannfoureleven}
\end{equation}
for $z$ inside the interval $-\overline{r}_j<z<\overline{r}_j$.
For $k>\overline{r}_j^{-1}$, set 
\begin{equation}
\nu_{j+1}(k)=\min\left\{\frac{1}{\beta_i(\frac{1}{k})};\hbox{   $i=1,\ldots,\ell$}\right\}
\label{Neumankatj}
\end{equation}
and $\nu_{j+1}(k)\rightarrow\nu_{j+1}$ for $k\rightarrow\infty$. Since $\nu_{j+1}(k)$ is increasing with $k$ we conclude
\begin{equation}
\nu_{j+1}(k)\leq\nu_{j+1}.
\label{Neumanninequal}
\end{equation}
We take a min-max ordering for $\sigma(-\Delta_D)$ and suppose that the $j^{th}$ Dirichlet eigenvalue corresponds to an eigenfunction of nonzero mean. The eigenvalue is  denoted by $\delta_j^\ast$. The min-max principle together with the monotonicity of eigenvalues with respect to increasing $k$ delivers the inequality between the Bloch eigenvalues and  $\sigma(-\Delta_D)$:
\begin{equation}
\lambda_j(k,\alpha)\leq \delta_j^\ast, \hbox{   for   $\alpha\in Y^\ast$,   and   $k>0$}.
\label{monotonego}
\end{equation}
Application of \eqref{minmax} and \eqref{monotonego} gives
\begin{eqnarray}
\lambda_m(k,\alpha)\leq\delta^\ast_j,\hbox{   for   $m\leq j$}& \hbox{and}& \nu_{j+1}(k)\leq\lambda_m(k,\alpha), \hbox{   for   $j+1\leq m$},
\label{intervallsss}
\end{eqnarray}
for every $\alpha\in Y^\ast$ and it is clear that a band gap opens in the Bloch spectrum when $\delta^\ast_j<\nu_{j+1}(k)$ or equivalently when
\begin{equation}
|\nu_{j+1}(k)-\nu_{j+1}|<|\delta_j^\ast-\nu_{j+1}|.
\label{openagapp}
\end{equation}
We apply \eqref{errorestNeumann}, \eqref{Neumanninequal}  and Remark \ref{Uniformini} to get
\begin{equation}
|\nu_{j+1}(k)-\nu_{j+1}|<\frac{\nu_{j+1}^2 d_j}{k\overline{r}_j-1},
\label{theestimateupperNeumann}
\end{equation}
and the theorem follows for all $k$ that satisfy
\begin{equation}
\frac{\nu_{j+1}^2 d_j}{k\overline{r}_j-1}<|\delta_j^\ast-\nu_{j+1}|.
\label{theestimateupperNeumann}
\end{equation}

\end{proof}

We now establish Theorem  \ref{passbandgeneral}.
\begin{proof} From the min-max formulation we have that $\lambda_j(k,0)$ and $\lambda_j(k,\alpha)$ are increasing with $k$ and from \cite{HempelLienau} (or \cite{RobertRobert1}) we have
\begin{equation}
\lim_{k\rightarrow\infty}\lambda_j(k,0)=\nu_j\hbox{   hence   } \lambda_j(k,0)\leq\nu_j, \hbox{   for $k>1$}
\label{lambdazerotomu}
\end{equation}
and
\begin{equation}
\lim_{k\rightarrow\infty}\lambda_j(k,\alpha)=\delta^\ast_j\hbox{   hence   } \lambda_j(k,\alpha)\leq\delta^\ast_j, \hbox{   for $k>1$}
\label{lambdaalphatomu}
\end{equation}

With this in mind observe that if $|\delta^\ast_j-\lambda_j(k,\alpha)|<|\delta^\ast_j-\nu_j|$ then
\begin{equation}
\sigma(L_k)\supset\left[\nu_j,\lambda_j(k,\alpha)\right].
\label{explicitpassbandinclusiongenerala}
\end{equation}
To proceed we estimate the difference $|\lambda_j(k,\alpha)-\delta^\ast_j|$. Set $\hat{d}=d_j$ with
$$d_j=\frac{1}{2}dist\left(\{(\delta_{j}^\ast)^{-1}\},\sigma(-\Delta_D)\setminus\{(\delta_{j}^\ast)^{-1}\}\right)$$ 
and $r^\ast_j$ given by \eqref{radiusalphanotzero} with $\hat{d}=d_j$.  Apply Theorem \ref{separationandraduus-alphanotzero} noting that $\delta_j^\ast$ is simple so that  $\beta^\alpha_j(z)$  has  series representation given by
\begin{equation}
{\beta}^\alpha_j(z) = (\delta_{j}^{\ast})^{-1} + \sum_{n=1}^\infty z^n\beta^n_j(z),
\label{seriesforalphafourelevenproof}
\end{equation}
for $z$ inside the interval $-{r}^\ast_j<z<{r}^\ast_j$.
For $k>{r}^\ast_j$ we have $\lambda_j(k,\alpha)=\frac{1}{\beta^\alpha_j(k^{-1})}$ and 
we apply \eqref{errorestalpha} 
to get
\begin{equation}
|\delta^\ast_j-\lambda_j(k,\alpha)|<\frac{(\delta_{j}^\ast)^2 d_j}{k{r}_j^\ast-1}.
\label{theestimateupperBloch}
\end{equation}
The persistence of band structure described by  \eqref{explicitpassbandinclusiongenerala} follows for a fixed $\alpha\in Y^\ast$ for all $k$ that satisfy
\begin{equation}
\frac{(\delta_{j}^\ast)^2 d_j}{kr^\ast_j-1}<|\delta_j^\ast-\nu_j|.
\label{thesufficiancyBloch}
\end{equation}
We maximize $r^\ast_j$ over $\alpha\in[-\pi,\pi]^d$ to find that it is attained for $|\alpha|^2=d\pi^2$ and the maximum is  $r^\ast_j=\underline{r}_j$.   For this choice we recover the persistence of band structure described by \eqref{explicitpassbandinclusiongeneral} and \eqref{explicitpassbandgeneral}.
\end{proof}

\section{Derivation of the Convergence Radius, Separation of Spectra and Error Estimates}
\label{derivation}

Here we prove Theorems \ref{separationandraduus-Neumann} and  \ref{errorestimatesforexpansions}. The Theorems \ref{separationandraduus-alphanotzero} and \ref{separationandraduus-alphazero} parts 1 through 3 and explicit convergence radii for power series representation for the eigenvalue group are proved in \cite{RobertRobert1}. To begin, we  recall that the Neumann series \eqref{foursix} and consequently \eqref{Project1} and \eqref{foureleven} converge provided that
\begin{equation}
\label{tenone}
\| (A(z) - A(0))R(\zeta,0) \|_{\mathcal{L}[L^2_{0}(Y);L^2_{0}(Y)]} <1.
\end{equation}
With this in mind we follow \cite{RobertRobert1} and  compute an explicit upper bound $B(z)$ and identify a neighborhood of the origin on the complex plane for which
\begin{equation}
\label{tenoneb}
\| (A(z) - A(0))R(\zeta,0) \|_{\mathcal{L}[L^2_{0}(Y);L^2_{0}(Y)]} <B(z)<1,
\end{equation}
holds for $\zeta\in\Gamma$.
The inequality $B(z)<1$ will be used first to derive a lower bound on the radius of convergence of the power series expansion of the eigenvalue group about $z=0$. It will then be used to provide a lower bound on the neighborhood of $z=0$ where properties 1 through 3 of Theorem \ref{separationandraduus-Neumann} hold.

We have the basic estimate given by
\begin{eqnarray}
\label{tenonedouble}
&&\| (A(z) - A(0))R(\zeta,0) \|_{\mathcal{L}[L^2_{0}(Y);L^2_{0}(Y)]}\leq \\
&&\| (A(z) - A(0))\|_{\mathcal{L}[L^2_{0}(Y);L^2_{0}(Y)]}\|R(\zeta,0) \|_{\mathcal{L}[L^2_{0}(Y);L^2_{0}(Y)]}.\nonumber
\end{eqnarray}
Here $\zeta\in\Gamma$ as defined in Theorem \ref{separationandraduus-alphanotzero} and elementary arguments deliver the estimate
\begin{eqnarray}
\label{tenonedoubleRz}
\|R(\zeta,0) \|_{\mathcal{L}[L^2_{0}(Y);L^2_{0}(Y)]}\leq \hat{d}^{-1},
\end{eqnarray}
where $\hat{d}$ is given by \eqref{dist}.

Next we estimate $\| (A(z) - A(0))\|_{\mathcal{L}[L^2_{0}(Y);L^2_{0}(Y)]} $. Denote the energy seminorm of  $u$ by \begin{equation}
\| u \|= \| \nabla u \|_{L^2(Y)}.
\end{equation}
To proceed we introduce the standard Poincare and Green's function estimates:
\begin{lemma}
For $u\in\mathcal{H}$
\begin{equation}
\label{alpha-poincare}
\| u \|^2_{L^2(Y)} \leq \lambda_N^{-1}\|u\|^2,
\end{equation}
and for $v\in L^2_0(Y)$
\begin{eqnarray}
||-\Delta_N^{-1}||\leq\lambda_N^{-1/2}\|v\|_{L^2(Y)}
\label{spectralbounddd}
\end{eqnarray}
\label{poincarealpha}
where $\lambda_N$ is the first nonzero Neumann eigenvalue for the period $Y$, $\lambda_N=\pi^2$ for $Y=(0,1]^d$.
\end{lemma}

For any $v \in L^2_{0}(Y)$, we  apply \eqref{alpha-poincare} to find
\begin{eqnarray}
\label{tenfive}
&&\| (A(z) - A(0)) v\|_{L^2(Y)} \nonumber\\
&&\leq  |\lambda_N|^{-1/2}\| (A(z) - A^{\alpha}(0)) v\|\\
&&= |\lambda_N |^{-1/2}\| ((T_k)^{-1} - (T_0)^{-1})(-\Delta_N)^{-1} v\|\nonumber\\
&&\leq |\lambda_N |^{-1/2}\| ((T_k)^{-1} - P_2)\|_{\mathcal{L}[\mathcal{H};\mathcal{H}]} \|-\Delta_N^{-1} v\|.\nonumber
\end{eqnarray}
Applying  \eqref{spectralbounddd} and \eqref{tenfive} delivers the upper bound:
\begin{equation}
\label{tenten}
\| (A(z) - A(0)) \|_{\mathcal{L}[L^2_{0}(Y);L^2_{0}(Y)]}  \leq  \lambda_N^{-1}\| ((T_k)^{-1} - P_2)\|_{\mathcal{L}[\mathcal{H};\mathcal{H}]}.
\end{equation}

The next step is to obtain an upper bound on $\| ((T_k)^{-1} - P_2)\|_{\mathcal{L}[\mathcal{H};\mathcal{H}]}$. For all $v \in \mathcal{H}$, we have

\begin{equation}
\label{teneleven}
\frac{\| ((T_k)^{-1} - P_2)v\|}{\| v \|} \leq |z|\{w_0 + \sum \limits_{i=1}^{\infty}w_i |(1/2 + \mu_i) + z(1/2-\mu_i)|^{-2}\}^{1/2},
\end{equation}
where $w_0=\|P_1 v\|^2/\|v\|^2$,  $w_i=\|P_i v\|^2/\|v\|^2$, and $w_0+\sum_{i=1}^\infty w_i=1$.
So maximizing the right hand side is equivalent to calculating

\begin{equation}
\begin{array}{lcl}
\max \limits_{w_0+\sum w_i = 1} \{w_0 + \sum \limits_{i=1}^{\infty}w_i |(1/2 + \mu_i) + z(1/2-\mu_i)|^{-2}\}^{1/2}\\
\\
 =  \sup \{1, |(1/2 + \mu_i) + z(1/2-\mu_i)|^{-2}\}^{1/2}.
\end{array}
\end{equation}
Thus we maximize the function
\begin{equation}
f(x) = |\frac{1}{2} + x + z(\frac{1}{2} -x)|^{-2}
\end{equation}
over $x \in [\mu^\ast, 1/2]$ for $z$ in a neighborhood about the origin.  Let $Re(z)=u$, $Im(z)=v$ and we write
\begin{equation}
\begin{array}{lcl}
f(x) & = & |\frac{1}{2} + x + (u+iv)(\frac{1}{2} -x)|^{-2}\\
\\ & = & ((\frac{1}{2}+x+u(\frac{1}{2}-x))^2 + v^2(\frac{1}{2}-x)^2)^{-1}\\
\\
& \leq & (\frac{1}{2}+x+u(\frac{1}{2}-x))^{-2} = g(Re(z),x)\text{,}
\end{array}
\end{equation}
to get the bound
\begin{equation}
\label{tenfifteen}
\| ((T_k)^{-1} - P_2)\|_{\mathcal{L}[\mathcal{H};\mathcal{H}]} \leq |z| \sup\{1, \sup \limits_{x \in [\mu^\ast, 1/2]} g(u,x)\}^{1/2}.
\end{equation}

We now examine the poles of $g(u,x)$ and the sign of its partial derivative $\partial_x g(u,x)$ when $|u|<1$.  If $Re(z)=u$ is fixed, then $g(u,x) = ((\frac{1}{2} + x) + u(\frac{1}{2} - x))^{-2}$ has a pole when $(\frac{1}{2} + x) + u(\frac{1}{2} - x)=0$.
For $u$ fixed this occurs when
\begin{equation}
\hat{x}=\hat{x}(u)= \frac{1}{2} \left(\frac{1+u}{u-1}\right).
\end{equation}
On the other hand, if $x$ is fixed, $g$ has a pole at
\begin{equation}
u= \frac{\frac{1}{2} + x}{x - \frac{1}{2}}.
\end{equation}
The sign of $\partial_x g$ is determined by the formula
\begin{equation}
\label{teneighteen}
\begin{array}{lcl}
\partial_x g(u,x) & = & {N}/{D} \text{,}
\end{array}
\end{equation}
where $N=-2(1-u)^2x-(1-u^2)$ and $D := ((\frac{1}{2} + x) + u(\frac{1}{2} - x))^4 \geq 0$.  Calculation shows that $\partial_x g<0$ for $x>\hat{x}$, i.e. $g$ is decreasing on $(\hat{x},\infty)$.  Similarly, $\partial_x g>0$ for $x<\hat{x}$ and $g$ is increasing on $(-\infty, \hat{x})$.\\

Now we identify all $u=Re(z)$ for which $\hat{x}=\hat{x}(u)$ satisfies 
\begin{equation}
\hat{x} < \mu^\ast < 0 \text{.}
\end{equation}
Indeed for such $u$, the function $g(u,x)$ will be decreasing on $[\mu^\ast, 1/2]$, so that $g(u,\mu^\ast) \geq g(u,x)$ for all $x \in [\mu^\ast, 1/2]$, yielding an upper bound for \eqref{tenfifteen}.
\begin{lemma}
\label{identifyu}
The set $U$ of $u \in \mathbb{R}$ for which $-\frac{1}{2} < \hat{x}(u) < \mu^\ast < 0$ is given by
$$U := [z^*, 1]$$
where
$$-1\leq z^*:=\frac{\mu^\ast+\frac{1}{2}}{\mu^\ast-\frac{1}{2}}<0.$$
\end{lemma}
\begin{proof}
Note first that $\mu^\ast=\inf_{i\in\mathbb{N}}\{\mu_i\}\leq 0$ follows from the fact that zero is an accumulation point for the sequence $\{\mu_i\}_{i\in\mathbb{N}}$ so it follows that $-1\leq z^*$.
Noting $\hat{x} = \hat{x}(u) = \frac{1}{2} \frac{u+1}{u-1}$, we invert and write
\begin{equation}
u = \frac{\frac{1}{2} + \hat{x}}{\hat{x} - \frac{1}{2}}.
\end{equation}
We now show that
\begin{equation}
  z^*\leq u\leq 1
\end{equation}
for $ \hat{x} \leq \mu^\ast$.  Set $h(\hat{x}) = \frac{\frac{1}{2} +\hat{x}}{\hat{x}-\frac{1}{2}}$.  Then
\begin{equation}
h'(\hat{x}) = \frac{-1}{(\hat{x}-\frac{1}{2})^2} \text{,}
\end{equation}
and so $h$ is decreasing on $(-\infty, \frac{1}{2})$.  Since $\mu^\ast<\frac{1}{2}$, $h$ attains a minimum over $(-\infty, \mu^\ast]$ at $x=\mu^\ast$.  Thus $\hat{x}(u) \leq \mu^\ast$ implies
\begin{equation}
z^*=\frac{\mu^\ast+\frac{1}{2}}{\mu^\ast - \frac{1}{2}} \leq u\leq 1
\end{equation}
as desired.
\end{proof}

Combining Lemma \ref{identifyu} with inequality \eqref{tenfifteen}, noting that $-|z|\leq Re(z) \leq |z|$ and on rearranging terms we obtain the following corollary.
\begin{corollary}
\label{boundAz}
For $|z| < |z^*|$:
\begin{equation}
\| (A(z) - A(0)) \|_{\mathcal{L}[L^2_{0}(Y);L^2_{0}(Y)]}  \leq \lambda_N^{-2} |z| (-|z|-z^*)^{-1}(\frac{1}{2}-\mu^\ast)^{-1}.
\end{equation}
\end{corollary}
From Corollary \ref{boundAz}, \eqref{tenonedouble}, and  \eqref{tenonedoubleRz}   we easily see that
\begin{eqnarray}
\label{tentwentyseven}
\| (A(z) - A(0))R(\zeta,0) \|_{\mathcal{L}[L^2_{0}(Y);L^2_{0}(Y)]} \leq
B(z)=\lambda_N^{-2} |z| (-|z|-z^*)^{-1}(\frac{1}{2}-\mu^\ast)^{-1}\hat{d}^{-1}. \label{product}
\end{eqnarray}
A straight forward calculation shows that $B(z)<1$ for
\begin{equation}
\label{thezridenty}
|z| < r^*:= \frac{\lambda_N^2\hat{d}|z^*|}{\frac{1}{\frac{1}{2} - \mu^\ast} + \lambda_N^2\hat{d}}.
\end{equation}
Since $r^* < |z^*|$ we have established that the Neumann series \eqref{foursix} and consequently \eqref{Project1} and \eqref{foureleven} converge for $|z|<r^\ast$.

Now we establish properties 1 through 3 of Theorem \ref{separationandraduus-Neumann}.
First note that inspection of  \eqref{foursix} shows that if \eqref{tenone} holds and if $\zeta\in\mathbb{C}$ belongs to the resolvent of $A(0)$  then it also belongs to the resolvent of $A(z)$. Since \eqref{tenone} holds for $\zeta\in\Gamma$ and $|z|<r^*$, property 1 of Theorem \ref{separationandraduus-Neumann} follows. Formula \eqref{Project1} shows that $P(z)$ is analytic in a neighborhood of $z=0$ determined by the condition that \eqref{tenone}  holds for $\zeta\in\Gamma$. The set $|z|<r^*$ lies inside this neighborhood
and property 2 of Theorem \ref{separationandraduus-Neumann} is proved. The isomorphism expressed in property 3 of Theorem \ref{separationandraduus-Neumann}  follows directly from Lemma 4.10
(\cite{KatoPerturb}, Chapter I \S 4) which is also valid in a Banach space. The uniform convergence of the series representation given by property 4 of Theorem \ref{separationandraduus-Neumann} follows from the error estimates presented in part 3 of Theorem \ref{errorestm}. Property 4 of Theorems \ref{separationandraduus-alphanotzero} and \ref{separationandraduus-alphazero} follow from the error estimates given by parts 1 and 2 of Theorem \ref{errorestm}.

Theorem \ref{errorestm} part 3 follows once we establish Cauchy-like inequalities for the coefficients $\beta_i^n(z)$ appearing in \eqref{foureleven}  given by
\begin{equation}
\label{cauchyinequall}
\left | \beta^n_i(z) \right | \leq \hat{d}(r^*)^{-n},
\end{equation}
for $|z|<r^*$.    From this it is evident that we can recover the estimates
\begin{equation}
\left |\hat{\beta}^{\alpha}(z) - \sum \limits_{n = 0}^{p} z^n \beta^{\alpha}_n \right | \leq \sum \limits_{n=p+1}^{\infty} |z|^n |\beta^{\alpha}_n| \leq \frac{\hat{d}|z|^{p+1}}{(r^*)^p(r^* - |z|)}\text{,}
\end{equation}
for $|z|<r^*$ and property 3 of \ref{errorestm} is established.

Now we establish \eqref{cauchyinequall}. Applying \eqref{product} and noting that $|\zeta-\beta_0|=\hat{d}$ on $\Gamma$  one obtains the estimate
\begin{equation}
|{\beta}^n_i(z)| = \hat{d}\Vert\mathcal{N}(\zeta,z)\Vert^n_{\mathcal{L}[L^2_0(Y);L^2_0(Y)]}\leq\hat{d}(\lambda_N^{-2}(|z^*|-|z|)^{-1}(\frac{1}{2}-\mu^\ast)^{-1}\hat{d})^n.
\label{fourtwelveproof}
\end{equation}
for $|z|<r^\ast$. Note that the righthand side of \eqref{fourtwelveproof} is increasing with $|z|\leq r^\ast<|z^\ast|$. The righthand side is maximized for $|z|=r^\ast$ and we recover \eqref{cauchyinequall} on applying \eqref{thezridenty}. We conclude noting that parts 1 and 2 of Theorem \ref{errorestm} follow identical arguments using Corollary 12.3  of \cite{RobertRobert1}.

Now we show that the operator $B(k)$ introduced in section \ref{bandstructure} is bounded and compact.
\begin{theorem}
\label{bounded}
The operator $B(k): L^2_{0}(Y) \longrightarrow \mathcal{H}$ is  bounded for $k\not\in Z$.
\end{theorem}
Observe for $v\in L^2_0(Y)$ that
\begin{eqnarray}
\Vert B(k) v\Vert&=&\vert T_k^{-1}(-\Delta_N)^{-1} v\Vert\nonumber\\
&\leq& \| (T_k^{-1}\|_{\mathcal{L}[\mathcal{H};\mathcal{H}]} \|-\Delta_N^{-1} v\|\nonumber\\
&\leq& \lambda_N^{-1/2}\| T_k^{-1}\|_{\mathcal{L}[\mathcal{H};\mathcal{H}]} \Vert v\Vert_{L^2(Y)},
\label{operatorfield}
\end{eqnarray}
where the last inequality follows from \eqref{spectralbounddd}. The upper estimate on $\| T_k^{-1}\|_{\mathcal{L}[\mathcal{H};\mathcal{H}]} $ is obtained from
\begin{equation}
\label{tenelevenpart2}
\frac{\| T_k^{-1}v\|}{\| v \|} \leq \{|z|\hat{w}+\tilde{w}+|\sum \limits_{i=1}^{\infty}w_i |(1/2 + \mu_i) + z(1/2-\mu_i)|^{-2}\}^{1/2},
\end{equation}
where $\hat{w}=\|P_1 v\|^2/\|v\|^2$=, $\tilde{w}=\|P_2v\|^2/\|v\|^2$, $w_i=\|P_i v\|^2/\|v\|^2$. Since $\hat{w}+\overline{w}+\sum_{i=1}^\infty w_i=1$
one recovers the upper bound
\begin{equation}
\label{tenelevenpart2}
\frac{\| T_k^{-1}v\|}{\| v \|} \leq M,
\end{equation}
where
\begin{equation}
M= \max\{1, |z|, \sup_{i} \{ |(1/2 + \mu_i) + z(1/2-\mu_i)|^{-1}\}\},
\label{summupperbound}
\end{equation}
and the proof of Theorem \ref{bounded} is complete.
\begin{remark}
The Poincare inequality \eqref{alpha-poincare}  together with Theorem \ref{bounded} show that
$B(k)$ is a bounded linear operator mapping $L^2_0(Y)$ into itself. The compact embedding of $\mathcal{H}$ into $L^2_0(\alpha,Y)$ shows the operator is compact on $L^2_0(Y)$. 
\label{compact2}
\end{remark}

We have the following theorem.
\begin{theorem}
$A(z)$ is compact, self adjoint and bounded on $L^2_0(Y)$ for $z$ real and $z\not\in S$.
\label{Selfadjoint}
\end{theorem}
The compactness and boundedness of $A(z)$ follow from Theorem  \ref{bounded}.
To see that $A(z)$ is self adjoint we write $(A(z)u,v)$ for $u$ and $v$ in $L^2_0(Y)$ and apply the formula 
\begin{equation}
\label{aformula}
A(z) = (zP_1 + P_2 + z\sum \limits_{-\frac{1}{2} < \mu_i < \frac{1}{2}} [(1/2 + \mu_i) + z(1/2-\mu_i)]P_{\mu_i})(-\Delta_N)^{-1} \text{.}
\end{equation}
We have 
\begin{eqnarray}
(A(z)u,v)&=&z(P_1(-\Delta_N)^{-1}u,v)+(P_2(-\Delta_N)^{-1}u,v)+\nonumber\\
&+& z\sum \limits_{-\frac{1}{2} < \mu_i < \frac{1}{2}} [(1/2 + \mu_i) + z(1/2-\mu_i)](P_{\mu_i}(-\Delta_N)^{-1}u,v).
\label{innerprod}
\end{eqnarray}
For $v\in\mathcal{H}$ the projections $P_{\mu_i}v$ are of the form
\begin{equation}
P_{\mu_i}v(x)=\sum_{\ell=1}^{m_i}\langle\Psi^i_\ell,v\rangle\Psi^i_\ell(x),\hbox{   for    $x\in Y$},
\label{innerprodproject}
\end{equation}
where $\{\Psi^i_\ell\}_{\ell=1}^{m_i}$ is an orthonormal basis for the subspace $P_{\mu_i}(\mathcal{H})$.
The projections $P_1$ and $P_2$ are defined similarly.  Without loss of generality we show that  $P_{\mu_i}(-\Delta_N)^{-1}$ is a self adjoint operator on $L^2_0(Y)$. For $p$ and $q$ in $L^2_0(Y)$ we apply \eqref{bigneumannident} to write
\begin{eqnarray}
&&(P_{\mu_i}(-\Delta_N)^{-1}p,q)=\sum_{\ell=1}^{m_i}\langle\Psi^i_\ell,(-\Delta_N)^{-1}p\rangle(\Psi^i_\ell,q)\nonumber\\
&&=\sum_{\ell=1}^{m_i}(\Psi^i_\ell,p)(\Psi^i_\ell,q).
\label{innerproddemo}
\end{eqnarray}
A similar argument holds for all terms in \eqref{innerprod} and we conclude that $A(z)$ is self adjoint for real $z$.

\section{Reciprocal relation and applications}
\label{concludingsection}
We introduce a reciprocal relation for both Bloch and Neumann spectra and use it to understand band structure for  crystals with coefficient  $a=\frac{1}{k}<1$ inside $\Omega$ and $a=1$ outside.
We denote the eigenvalue  associated with a choice of coefficient $a(x)=a^{in}$ for points $x\in D$ and $a=a^{out}$ for points $x\in Y\setminus D$ by $\omega^2=\omega^2(a^{in},a^{out})$. The spectral problem is given by the solution $u$ of
\begin{equation}
a^{in} \int_{D} \nabla u(x) \cdot \nabla \bar{v}(x)dx+a^{out}\int_{Y \setminus D} \nabla u(x) \cdot \nabla \bar{v}(x)dx=\omega^2\int_Y\,u(x)\bar{v}(x)\,dx,
\label{phaseses}
\end{equation} 
for all test functions $v$. The Bloch spectrum is associated with $u$ and $v$ in $H_{\#}^1(\alpha,Y)$ and the Neumann spectrum is associated with  $u$ and $v$ in $\mathcal{H}$. Now let $t$ be scalar and the spectrum satisfies the homogeneity property
\begin{equation}
\omega^2(ta^{in},a^{out})=t\omega^2(a^{in},t^{-1}a^{out}),
\label{phasescaling}
\end{equation}
so for $a=\frac{1}{k}$ in $D$ and $1$ in $Y\setminus D$ we have the reciprocal relation
\begin{equation}
\omega^2(\frac{1}{k},1)=\frac{1}{k}\omega^2(1,k).
\label{phaseinterchange}
\end{equation}

The reciprocal relation  \eqref{phaseinterchange} provides the relation between the band structure for the operator $L_k$ described by \eqref{Eigen0} and  $$\tilde{L}_k=-\nabla\cdot((1- \chi_{\scriptscriptstyle{\Omega}})+k^{-1} \chi_{\scriptscriptstyle{\Omega}})\nabla$$
and is given by 
\begin{equation}
\sigma(\tilde{L}_k)=\frac{1}{k}\sigma(L_{k}).
\label{flippedspectra}
\end{equation}

As an application we return to the photonic crystal given by the periodic dispersion of $N$ disks each separated by a minimum distance as described section \ref{introduction}. We suppose that the dielectric constant inside each disk is now greater than $1$ and given by $k$ while the surrounding material has dielectric constant $1$. For this case we apply \eqref{flippedspectra} together with Theorems \ref{stopband} and \ref{passband} to recover the following theorem on existence of band gaps and persistence of spectral bands for H-polarized modes.
\begin{theorem}{Opening a band gap}\\
Given $\delta_{0j}^\ast$ define the  the set  $\sigma_N^+$ to be elements $\nu\in\sigma_N$ for which $\nu>\delta_{0j}^\ast$. The element in $\sigma_N^+$ closest to $\delta_{0j}^\ast$ is denoted by $\nu_{j+1}$. Set $d_j$ according to
\begin{equation}
d_j=\frac{1}{2}\min\left\{|\nu_{j+1}^{-1}-\nu^{-1}|;\hbox{   $\nu\in\sigma_N$}\right\}.
\label{disttospectrum2}
\end{equation}
We define $\overline{r}_j$ to be
\begin{equation}
\label{upperr-j2}
\overline{r}_j = \frac{\pi^2d_j(b^2-a^2)}{(b^2+a^2) + \pi^2d_j(b^2+3a^2)}.
\end{equation}
Then one has the band gap
\begin{equation}
\sigma(\tilde{L}_k)\cap\frac{1}{k}\left(\delta_{0j}^\ast,\nu_{j+1}(1-\frac{\nu_{j+1}d_j}{k\overline{r}_j-1})\right)=\emptyset
\label{explicitgap2}
\end{equation}
if
\begin{equation}
k>\overline{k}_j=\overline{r}_j^{-1}\left(1+\frac{d_j\nu_{j+1}}{1-\frac{\delta_{0j}^\ast}{\nu_{j+1}}}\right).
\label{explicitbandgap2}
\end{equation}
\label{stopband2}
\end{theorem}
Next we provide an explicit condition on $k$  sufficient for the persistence of a spectral band together with explicit formulas describing its location and bandwidth. 

\begin{theorem}{Persistence of passbands}\\
Given  $\delta_{0j}^\ast$ define the  the set  $\sigma_N^-$ to be elements $\nu\in\sigma_N$ for which $\nu<\delta_{0j}^\ast$. The element in $\sigma_N^-$ closest to $\delta_{0j}^\ast$ is denoted by $\nu_{j}$.  Set $d_j$ according to
\begin{equation}
d_j=\frac{1}{2}\min\left\{|(\delta^\ast_{j0})^{-1}-\delta^{-1}|;\hbox{   $\delta\in\sigma(-\Delta_D)$}\right\}.
\label{disttospectrumDirichlet2}
\end{equation}
Define $\underline{r}_j$ to be
\begin{equation}
\label{lowerr-j2}
\underline{r}_j = \frac{2\pi^2d_j(b^2-a^2)}{(b^2+a^2) + 2\pi^2d_j(b^2+3a^2)}.
\end{equation}
Then one has a passband in the vicinity of $\delta_{0j}^\ast$ and
\begin{equation}
\sigma(\tilde{L}_k)\supset\frac{1}{k}\left[\nu_j,\delta_{0j}^\ast(1-\frac{\delta_{0j}^\ast d_j}{k\underline{r}_j-1})\right]
\label{explicitband2}
\end{equation}
if
\begin{equation}
k>\underline{k}_j=\underline{r}_j^{-1}\left(1+\frac{d_j\delta_{0j}^\ast}{1-\frac{\nu_j}{\delta_{0j}^\ast}}\right).
\label{explicitpassband2}
\end{equation}
\label{passband2}
\end{theorem}

Observe that since $1/k<1$ Theorems \ref{stopband2} and \ref{passband2} provide qualitative criteria  for sub wavelength control of band structure for $k>\overline{k}$. We conclude and point out that the results developed here provide rigorous criteria based on geometry and material properties for opening band gaps in both 2 and 3 dimensional periodic materials, see Theorems   \ref{stopband},  \ref{passband},    \ref{stopbandgeneral}  \ref{passbandgeneral},  \ref{stopband2}, and   \ref{passband2}.



\section*{Acknowledgements}
This research is supported by AFOSR MURI Grant FA9550-12-1-0489 administered through the University of New Mexico, NSF grant DMS-1211066, and NSF EPSCOR Cooperative Agreement No. EPS-1003897 with additional support from the Louisiana Board of Regents.

\bibliographystyle{plain}
\bibliography{Photonicrefs1}

\begin{thebibliography}{10}

\bibitem{AmmariKang1}
H.~Ammari, H.~Kang, S.~Soussi, and H.~Zribi.
\newblock Layer potential techniques in spectral analysis. part 2: Sensitivity
  analysis of spectral properties of high contrast band-gap materials.
\newblock {\em Multiscale Model. Simul.}, 5:646--663, 2006.

\bibitem{BendsoeSigmund}
M.~Bendsoe and O.~Sigmund.
\newblock {\em Topology Optimization Theory, Methods, and Applications}.
\newblock Springer Verlag, Berlin, Heidelberg, New York, 2004.

\bibitem{BergmanC}
D.J. Bergman.
\newblock The dielectric constant of a composite material - a problem in
  classical physics.
\newblock {\em Physics Reports}, 43:377--407, 1978.

\bibitem{BergmanES}
D.J. Bergman.
\newblock The dielectric constant of a simple cubic array of identical spheres.
\newblock {\em J. Phys. C.}, 12:4947--4960, 1979.

\bibitem{Bruno}
O.P. Bruno.
\newblock The effective conductivity of strongly heterogeneous composites.
\newblock {\em Proceedings of the Royal Society of London A: Mathematical and
  Physical Sciences}, 433:353--381, 1991.

\bibitem{CostabelBdryOps}
M.~Costabel.
\newblock Boundary integral operators on lipschitz domains: Elementary results.
\newblock {\em SIAM Journal of Mathematical Analysis}, 19(3):613--625, 1988.

\bibitem{CoxDobsonH}
S.J. Cox and D.C. Dobson.
\newblock Band structure optimization of two-dimensional photonic crystals in
  hpolarization.
\newblock {\em J. Comput. Phys.}, 158:214--224, 2000.

\bibitem{Elbert}
A.~Elbert.
\newblock Some recent results on the zeros of bessel functions and orthogonal
  polynomials.
\newblock {\em Journal of Computational and Applied Mathematics}, 133:65--83,
  2001.

\bibitem{FigKuch2}
A.~Figotin and P.~Kuchment.
\newblock Band-gap structure of the spectrum of periodic maxwell operators.
\newblock {\em Journal of Statistical Physics}, 74:447--455, 1994.

\bibitem{FigKuch3}
A.~Figotin and P.~Kuchment.
\newblock Band-gap structure of spectra of periodic dielectric and acoustic
  media. 1. scalar model.
\newblock {\em SIAM J. Appl. Math.}, 56:68--88, 1996.

\bibitem{FigKuch1}
A.~Figotin and P.~Kuchment.
\newblock Spectral properties of classical waves in high-contrast periodic
  media.
\newblock {\em SIAM J. Appl. Math.}, 58:683--702, 1998.

\bibitem{Friedlander}
L.~Friedlander.
\newblock On the density of states of periodic media in the large coupling
  limit.
\newblock {\em Communications in Partial Differential Equations}, 27:355--380,
  2002.

\bibitem{GoldenPap}
K.~Golden and G.~Papanicolaou.
\newblock Bounds for effective parameters of heterogeneous media by analytic
  continuuation.
\newblock {\em Commun, Math. Phys.}, 90:473--491, 1983.

\bibitem{HempelLienau}
R.~Hempel and K.~Lienau.
\newblock Spectral properties of periodic media in the large coupling limit.
\newblock {\em Communications in Partial Differential Equations},
  25:1445--1470, 2000.

\bibitem{J87}
S.~John.
\newblock Strong localization of photons in certain disordered dielectric
  superlattices.
\newblock {\em Phys. Rev. Lett.}, 58:2486--2489, 1987.

\bibitem{Joo}
I.~Joo.
\newblock On the control of a circular membrane.
\newblock {\em Acta Math. Hungar.}, 61:302--325, 1993.

\bibitem{Kang}
H.~Kang.
\newblock Layer potential approaches to interface problems.
\newblock In {\em Inverse Problems and Imaging: Panoramas et synth\'eses 44}.
  Soci\'et\'e Math\'ematique de France, 2013.

\bibitem{KaoOsherYablonovich2005}
C.Y. Kao, S.J. Osher, and E.~Yablonovitch.
\newblock Òmaximizing band gaps in two-dimensional photonic crystals by using
  level set methods.
\newblock {\em Appl. Phys. B}, 81:235--244, 2005.

\bibitem{TKato1}
T.~Kato.
\newblock On the convergence of the perturbation method, 1.
\newblock {\em Progr. Theor. Phys.}, 4:514--523, 1949.

\bibitem{TKato2}
T.~Kato.
\newblock On the convergence of the perturbation method, 2.
\newblock {\em Progr. Theor. Phys.}, 5:95--101, 1950.

\bibitem{KatoPerturb}
T.~Kato.
\newblock {\em Perturbation Theory for Linear Operators}.
\newblock Springer, Berlin Heidelberg, Germany, 1995.

\bibitem{Shapero}
D.~Khavinson, M.~Putinar, and H.~Shapiro.
\newblock Poincar\'e's variational problem in potential theory.
\newblock {\em Archive for Rational Mechanics and Analysis}, 185:143--184,
  2007.

\bibitem{Kuchment}
P.~Kuchment.
\newblock {\em Floquet Theory for Partial Differential Equations}.
\newblock Birkhauser Verlag, Basel, 1993.

\bibitem{RobertRobert1}
R.~Lipton and R.~Viator Jr.
\newblock Bloch waves in crystals and periodic high contrast media.
\newblock {\em ESAIM Mathematical Modeling and Numerical Analysis}, Accepted
  June 16, 2016.

\bibitem{MiltonES}
R.C. McPhedran and G.W. Milton.
\newblock Bounds and exact theories for transport properties of inhomogeneous
  media.
\newblock {\em Applied Physics A.}, 26:207--220, 1981.

\bibitem{MenLeeFreundPeraireJohnson}
H.~Men, K.Y.K. Lee, R.M. Freund, J.~Peraire, and S.G. Johnson.
\newblock Robust topology optimization of three-dimensional photonic-crystal
  band-gap structures.
\newblock {\em Optics Express 22634}, 22(19), 2014.

\bibitem{Milton}
G.W. Milton.
\newblock {\em The Theory of Composites}.
\newblock Cambridge University Press, Cambridge, 2002.

\bibitem{OdehKeller}
F.~Odeh and J.B. Keller.
\newblock Partial differential equations with periodic coefficients and bloch
  waves in crystals.
\newblock {\em J. Math. Phys.}, 5:1499--1504, 1964.

\bibitem{ReedSimon}
M.~Reed and B.~Simon.
\newblock {\em Methods of Modern Mathematical Physics, Vol. IV: Analysis}.
\newblock Academic Press, New York, 1978.

\bibitem{Selden}
J.~Selden.
\newblock Periodic operators in high-contrast media and the integrated density
  of states function.
\newblock {\em Communications in Partial Differential Equations},
  30:1021--1037, 2005.

\bibitem{SzNagy}
B.~Sz.-Nagy.
\newblock Perturbations des transformations autoadjoints dans l\'espace de
  hilbert.
\newblock {\em Comment. Math. Helv.}, 19:347--366, 1946.

\bibitem{WangJensenSigmund2011}
F.~Wang, J.~S. Jensen, and O.~Sigmund.
\newblock Robust topology optimization of photonic crystal waveguides with
  tailored dispersion properties.
\newblock {\em J. Opt. Soc. Am. B}, 28:767--784, 2011.

\bibitem{Wilcox}
C.~Wilcox.
\newblock Theory of bloch waves.
\newblock {\em J. Analyse Math.}, 33:146--167, 1978.

\bibitem{Y}
E.~Yablonovitch.
\newblock Inhibited spontaneous emission in solid-state physics and
  electronics.
\newblock {\em Phys. Rev. Lett.}, 63:2059--2062, 1987.

\end{thebibliography}






\end{document}